\documentclass[11pt]{amsart}
\usepackage{url}
\usepackage[cmtip, all]{xy}
\usepackage[utf8]{inputenc}
\usepackage[pdftex,
     pdfview=XYZ,
     pdfstartview=FitV,
     pdffitwindow,
  hyperfootnotes=false,
  raiselinks=true,
  hypertexnames=false,
  colorlinks=true,
  bookmarks=true,
  bookmarksopen=false,
  anchorcolor=black,
  citecolor=black,
  linkcolor=black,
  filecolor=black,
  urlcolor=black]{hyperref}
\usepackage{amssymb,verbatim,upref}

\providecommand{\eprint}[1]{\href{http://arxiv.org/abs/#1}{\texttt{arXiv\string:\allowbreak#1}}}

\theoremstyle{plain}
\newtheorem{thm}{Theorem}[section]
\newtheorem{lemma}[thm]{Lemma}

\newtheorem{proposition}[thm]{Proposition}
\theoremstyle{definition}
\newtheorem{facts}[thm]{Facts}
\newtheorem{remark}[thm]{Remark}
\newtheorem{notations}[thm]{Notations}
\newtheorem{definition}[thm]{Definition}
\newtheorem{claim}[thm]{Claim}
\newtheorem{assumption}[thm]{Assumption}

\numberwithin{equation}{section}

\newcommand{\ga}[2]{\begin{gather}\label{#1}#2 \end{gather}}

\let\ra\rightarrow
\def\to{\mathchoice{\longrightarrow}{\rightarrow}{\rightarrow}{\rightarrow}}

\def\hto{\mathrel{\lhook\joinrel\to}}
\def\To#1{\mathchoice{\xrightarrow{\textstyle\kern4pt#1\kern3pt}}{\stackrel{#1}{\longrightarrow}}{}{}}
\def\Mto#1{\mathrel{\mapstochar\To{#1}}}
\def\Hto#1{\mathrel{\lhook\joinrel\To{#1}}}

\newcommand{\gr}{{\rm gr}}
\let\Sp\Spec
\newcommand{\sD}{{\mathcal D}}
\newcommand{\sE}{{\mathcal E}}
\newcommand{\sF}{{\mathcal F}}
\newcommand{\sH}{{\mathcal H}}
\newcommand{\sL}{{\mathcal L}}
\newcommand{\sN}{{\mathcal N}}
\newcommand{\sO}{{\mathcal O}}
\newcommand{\sR}{{\mathcal R}}
\newcommand{\sT}{{\mathcal T}}
\newcommand{\C}{{\mathbb C}}
\newcommand{\N}{{\mathbb N}}
\newcommand{\Q}{{\mathbb Q}}
\newcommand{\Z}{{\mathbb Z}}
\DeclareMathOperator{\Mod}{\mathsf{Mod}}
\DeclareMathOperator{\id}{id}
\DeclareMathOperator{\image}{im}

\let\wh\widehat
\let\hat\widehat
\let\emptyset\varnothing
\let\hat\widehat
\let\epsilon\varepsilon

\newcommand{\rb}{\mathrm{b}}

\DeclareMathOperator{\catD}{\mathsf{D}}
\DeclareMathOperator{\DR}{DR}

\newcommand{\Euler}{\mathrm{Eu}}
\newcommand{\rc}{\mathrm{c}}
\newcommand{\triv}{\mathrm{triv}}
\newcommand{\wsD}{\widetilde\sD}
\newcommand{\wsE}{\widetilde\sE}
\newcommand{\wsN}{{\widetilde\sN}}
\newcommand{\wsO}{\widetilde\sO}
\newcommand{\wsT}{{\widetilde\sT}}
\newcommand{\wVsE}{\widetilde{V^{-1}\sE}}

\newcommand{\wTheta}{\widetilde\Theta}
\newcommand{\wpartial}{\widetilde\partial}
\newcommand{\wxi}{\widetilde\xi}

\newcommand{\btheta}{{\boldsymbol{\theta}}}

\DeclareMathOperator{\Car}{Car}

\DeclareMathOperator{\Spencer}{Sp}

\title [Good lattices]{Good lattices of algebraic connections}
\author{Hélène Esnault \and Claude Sabbah}
\address{Freie Universität Berlin, Mathematik, Arnimallee 3, 14195, Berlin, Germany}
\email{esnault@math.fu-berlin.de}
\address{CMLS, École polytechnique, CNRS, Université Paris-Saclay,
F--91128 Palaiseau cedex,
France}
\email{claude.sabbah@polytechnique.edu}
\thanks{The first author is supported by the Einstein program}

\begin{document}
\begin{abstract}
We construct a logarithmic model of connections on smooth quasi-projective $n$-dimensional geometrically irreducible varieties defined over an algebraically closed field of characteristic $0$. It consists of a good compactification of the variety together with $(n+1)$ lattices on it which are stabilized by log differential operators, and compute algebraically de Rham cohomology. The construction is derived from the existence of good Deligne-Malgrange lattices, a theorem of Kedlaya and Mochizuki which consists first in eliminating the turning points. Moreover, we show that a logarithmic model obtained in this way, called a good model, yields a formula predicted by Michael Groechenig, computing the class of the characteristic variety of the underlying $\sD$-module
in the $K$-theory group of the variety.
\end{abstract}


\maketitle

\section{Introduction} \label{intro}
Let $U$ be a smooth quasi-projective geometrically irreducible variety of dimension $n$ defined over a characteristic $0$ field $k$. An open embedding $j: U\hto X$ is said to be a {\it good compactification} if $X$ is smooth projective and $D=X\setminus U$ is a strict normal crossings divisor. Here strict normal crossings divisor means that the irreducible components of $D_{\bar k}
$, where $\bar k\supset k$ is an algebraic closure and the lower index $_{\bar k}$ indicates the base change $\otimes_k \bar k$, are smooth and intersect transversally.

\begin{definition} \label{dfn:goodmodel}
Let $(\sE,\nabla)$ be vector bundle on $U$ with an integrable connection relative to $k$.
A tuple $$\big(X, (E_0, E_1,\ldots, E_n)\big)$$ is called a \emph{logarithmic model} of $\big(U, (\sE,\nabla)\big)$ if the following conditions are fulfilled.
\begin{itemize}
\item[0)] $j: U\hto X$ is a good compactification;
\item[1)] $E_0\subset E_1\subset \ldots \subset E_n $ is a tower of locally free lattices of $ j_*\sE$;
\item[2)] $\nabla(\Omega^i_X(\log D)\otimes_{\sO_X} E_i)\subset \Omega^{i+1}_X(\log D)\otimes_{\sO_X} E_{i+1}$;
\item[3)]For any effective divisor $\Delta$ with support in $D$, the embeddings of $k$-linear complexes
\ga{}{\Omega^\bullet_X(\log D) \otimes_{\sO_X} E_\bullet \subset
\Omega^\bullet_X(\log D)\otimes_{\sO_X} E_\bullet (\Delta) \subset
j_*(\Omega^\bullet_U\otimes_{\sO_U} \sE) \notag}
are quasi-isomorphisms.
\end{itemize}
\end{definition}

There is another natural definition.
\begin{definition} \label{dfn:verygoodmodel}
Let $(\sE,\nabla)$ be vector bundle on $U$ with an integrable connection relative to $k$.
A tuple $$\big(X, (E_0, E_1,\ldots, E_n)\big)$$ is called a \emph{good model} of $\big(U, (\sE,\nabla)\big)$ if the following conditions are fulfilled.
\begin{itemize}
\item[0)] $j: U\hto X$ is a good compactification;
\item[0')] $j$ resolves the turning points of $(\sE,\nabla)$;
\item[1')] $E_0$ is the Deligne-Malgrange lattice of $\sE$ (see Definition~\ref{defn:goodlattice} and Theorem~\ref{thm:mochizuki}) and~$E_{i+1}$ is defined inductively as the $\sO_X$-coherent subsheaf of $j_*\sE$ spanned by~$ E_i$ and $\Theta_X(\log D) \cdot E_i$.
\end{itemize}
\end{definition}
Here, $\Theta_X(\log D)$ is the sheaf of vector fields stabilizing~$D$. That good models exist is due to Kedlaya and Mochizuki. Our first main result is the following theorem.

\begin{thm} \label{thm:main}
Good models are logarithmic models.
\end{thm}
Our purpose is to prove that the lattices $E_i$ as described in 1') verify 1) and~3). It is performed in Section~\ref{sec:Ei} by constructing specific filtrations on the (log)-de Rham complexes, the graded of which are $\sO_X$-linear.

The Deligne-Malgrange lattice $E_0$ and its log derivatives $E_i$ ($i\geq1$) can also be used to compute the characteristic class of $j_*\sE$ in Grothendieck's $K$-group $K_0(\sO_X)$. Recall that, in \cite[\S6.1]{Lau83}, Laumon defines a group homomorphism
\[
\Car:K_0(\sD_X)\to K_0(\sO_{T^*X})
\]
 as follows. Let $F_\bullet\sD_X$ be the filtration of $\sD_X$ by the order of differential operators. By an \emph{$F\sD_X$-filtration of a $\sD_X$-module $M$}, we mean an increasing filtration $F_\bullet M$ such that $F_k\sD_X\cdot F_\ell M\subset F_{k+\ell}M$ for all $k,\ell\in\Z$. A \emph{coherent $F\sD_X$-filtration} is an $F\sD_X$-filtration such that $\gr^FM$ is coherent over $\gr^F\sD_X$. For any coherent $\sD_X$-module $M$, let $F_\bullet M$ be any coherent $F\sD_X$-filtration. Then,
denoting by $[\cdot]$ the class in $K_0$, one has
$$\Car M=[\gr^FM].$$
Composing with the restriction by the zero section $i:X\hto T^*X$, we obtain a group homomorphism
\[
Li^*\Car:K_0(\sD_X)\to K_0(\sO_X).
\]
Computing $Li^*\Car M$ amounts to computing the class of
$$\gr^F M\otimes^L_{\sO_{T^*X}}\sO_X \in K_0(\sO_X).$$ Our second main result confirms an expectation of Michael Groechenig.

\begin{thm}\label{thm:main2}
Let $(\sE,\nabla)$ be vector bundle on $U$ with an integrable connection relative to~$k$. Let
$\big(X, (E_0, E_1,\ldots, E_n)\big)$ be a good model of $\big(U, (\sE,\nabla)\big)$.
Then
\[
Li^*\Car(j_*\sE)=\sum_{i=0}^n(-1)^{n-i}[\omega_X^{-1}\otimes\Omega^i_X(\log D)\otimes E_i]\quad \text{in }K_0(\sO_X).
\]
\end{thm}

The proof rests on the two Appendices by the second named author. Appendix \ref{app:A} is
classical, see e.g.\ \cite{MHM}. Appendix \ref{app:B} is inspired by \cite{Wei17} and gives a criterion for exactness of the tensor product of the differential operators over the ones with logarithmic poles.
\subsubsection*{Acknowledgements} Michael Groechenig asked us whether there is a generalization of Deligne's theory of pairs of good lattices in higher dimension, and once the notion was defined and the good models were constructed, predicted Theorem~\ref{thm:main2}. It is a great pleasure to acknowledge his influence on our note which would not exist without his insight. We thank him profoundly. We thank the referee for remarks and comments which enabled us to improve our manuscript.

\section{The theorem of Kedlaya-Mochizuki}
\label{sec:mochizuki}

Sabbah's conjecture~\cite[Conj.\,2.5.1]{Sab00} over $k=\C$, stipulating the existence of a good {\it formal} structure of $(\sE, \nabla)$ at each closed point of $(X_0\setminus U)$ for a given good compactification $U\hto X_0$, after blow up and ramification, was solved by Mochizuki and Kedlaya, and for the latter, over any characteristic $0$ field. Mochizuki's theorem is summarized in \cite{Moc09} and contains more structure than Sabbah's initial conjecture. It yields a uniquely defined {\it lattice} with certain properties. See the definitions in {\it loc.\ cit.}\ 2.2.2, Remark~2.3, Conjecture~2.11, which is Sabbah's conjecture {\it enhanced} with the existence of good Deligne-Malgrange lattices after blow-up of a given good compactification, and ramification at the formal completion of each closed point at infinity, and see Theorem~2.12 in which Conjecture~2.11 is solved.

We use Kedlaya's more algebraic approach \cite[\S5.3]{Ked11}, that we now recall.
Let $j: U\hto X$ be a good compactification
and $x$ be a finite union of closed points in $ D$. We denote by $X_x$, resp. $D_x$ the spectrum of the semi-local ring $\sO_{X,x}$, resp. $\sO_{D,x}$ of $X$, resp. $D$ at $x$, by $\hat X_x$, resp.\ $\hat D_x$ the formal spectrum of the completion $\hat \sO_{X,x}$, resp.\ $\hat \sO_{D,x}$ with respect to the ideal of
$x$. We set
$$ U_x=X_x\setminus D_x, \ \ \hat U_x=\hat X_x\setminus \hat D_x.$$
and
$$ (\sE_x, \nabla_x) = \sO_{X_x}\otimes_{\sO_X} (j_*\sE , \nabla), \ \ (\hat \sE_x, \hat \nabla_x)=
\sO_{\hat X_x}\otimes_{\sO_X }(j_*\sE, \nabla).
$$
Likewise, for any sheaf $\sF$ of $\sO_X$-modules on $X$, we define $\sF_x$ and $\hat \sF_x$. We also denote by $\hat{\sF_x}$ an $ \sO_{\hat{X_x}}$-module, which is not necessarily defined over $\sO_{X_x}$, and similarly we use the symbol $\hat{(-_x)}$ for an
$\sO_{\hat{X_x}}$-morphism which does not necessarily descend to $\sO_{X,x}$.

If $K\supset k$ is any field extension, we denote by a lower index $_K$ the base change $\otimes_k K$.
Let $\bar k\supset k$ be the choice of an algebraic closure. Let $\tau: \bar k/\Z\to \bar k$ be the admissible section of the projection $\bar k \to \bar k/\Z$ (see Kedlaya's definition \cite[Def.\,5.3.2]{Ked11}), which for $k=\C$ is the one used by Deligne and is characterized by the property
that the real part of the image lies in $[0 \ 1)$.

\begin{definition} \label{defn:goodlattice}
Let $j: U \hto X$ be a good compactification of a smooth quasi-projective geometrically irreducible
variety defined over $k$, let $(\sE,\nabla)$ be a vector bundle with an integrable connection on $U$.
\begin{itemize}
\item[1)]
An {\it unramified good Deligne-Malgrange lattice} $E\subset j_* \sE $ is a lattice such that for every closed point $x\in D = X \setminus U$,
$$\hat E_x= \bigoplus_{\varphi\in \Phi} (L_\varphi \otimes R_{\varphi}), $$
where $R_{\varphi}$ is Deligne's extension of a regular singular connection associated to $\tau$, $\Phi\subset\sO(\hat U_x)$ is a finite set satisfying the properties in \cite[Def.\,3.4.6 (a) \&\ (b)]{Ked11}, and $L_\varphi$ is a purely irregular lattice of a connection of rank $1$, i.e.,
$L_\varphi$ is isomorphic to $\sO_{\hat X_x}$ with $\nabla(1)= d\varphi$. 
\item[2)] A lattice $ E\subset j_* \sE $ is {\it a good Deligne-Malgrange lattice} if, for every closed point $x\in D$,
there exists a finite Galois cover
$\widehat{h_x}: \hat{X'_{x'}} \to \hat X_x $, étale over $\hat U_x$, where $x'$ is a finite union of closed points, with
$ \hat{X'_{x'}} $ formally smooth, of group $G$, a $G$-invariant lattice $$\hat{E'_{x'}} \subset \widehat{ h_x}^* \hat \sE_x$$ such that
for all closed points $y$ in $x'$, the lattice $\widehat{(E'_{x'})}_y$ is an unramified good Deligne-Malgrange lattice and such that
$$\hat E_x = \hat{({E'_{x'}})}{}^{G}.$$
\end{itemize}
\end{definition}

\begin{remark}\mbox{}
\begin{enumerate}
\item
Recall \cite[Th.\,5.3.4]{Ked11} that if a good Deligne-Malgrange lattice $E$ exists on a good compactification $X$ of $U$, it is locally free, and the isomorphism classes of $E$ and~$\hat E_x$ are unique for any closed point $x\in D$. See also the proof of Lemma \ref{lem:Ei}.
\item
The Galois cover $\hat{h_x}$ in 2) could simply be a base change $K\supset k$ necessary to make the geometric irreducible components of $D$ rational over $K$.
\end{enumerate}
\end{remark}

\begin{definition} \label{defn:unram}
The formal integrable connection on $ \hat U_x$ is said to be {\it unramified } if it admits a good Deligne-Malgrange lattice where $\hat{h_x}$ is the base change by a field extension.
A connection $(\sE,\nabla)$ is said to be unramified if $(\hat \sE_x, \hat \nabla_x)$ is unramified for all closed points $x\in D$.
\end{definition}

\begin{thm}[\cite{Moc09}, Theorem~2.12 and Conjecture~2.11 for the definition used in Theorem~2.12, \cite{Moc11}, Section~2.7 and Theorem~16.2.1, \cite{Ked11}, Theorem~8.1.3, Hypotheses~8.1.1, Theorem 8.2.3]
\label{thm:mochizuki} Let $j: U\hto X_0$ be a good compactification of a smooth
variety of finite type defined over a field $k$ of characteristic $0$.
Let $(\sE,\nabla)$ be a vector bundle with a flat connection on $U$. Then there exists a proper birational map $X\to X_0 $, which is an isomorphism on $U$, such that $j: U\hto X$ is a good compactification and the locally free $j_*\sO_U$-sheaf $j_* \sE$ admits a
good Deligne-Malgrange lattice $E$.
\end{thm}

Mochizuki's proof is analytic, and holds for meromorphic connections on
complex projective varieties as well, while Kedlaya's proof is algebraic and rational over the field of definition. The latter does not insist on the projectivity of $X$ when $U$ is assumed to be quasi-projective. However, they both construct $X$ starting from a given $X_0$ as above. Taking $X_0$ to be projective, then their $X$, which is proper over $X_0$, admits a proper modification $X'\to X$ such that $U\hto{} X'$ is a good compactification and $X'$ is projective. Then one applies the covariance of Deligne-Malgrange lattices, see e.g.\ \cite[Rem.\,5.3.7, 2nd part]{Ked11}, where the formula should read $f_*\sE'_0\cap\sE=\sE_0$.
We remark that if $(\sE, \nabla)$ is regular singular, then $E$ is Deligne's lattice \cite[Prop.\,5.4]{Del70} and that, over the complex numbers, good Deligne-Malgrange lattices coincide with the canonical lattices constructed in \cite{Mal96}.

\section{Proof of Theorem \ref{thm:main}} \label{sec:Ei}
The purpose of this section is to prove Theorem~\ref{thm:main}. We fix an effective divisor $\Delta$ with support in $D$.

Let $\Theta_X(\log D)$ be the sheaf of $D$-logarithmic vector fields on $X$. Define inductively the sequence $E_i$ of $\sO_X$-coherent subsheaves of $j_*\sE$ by the following rule:
\begin{itemize}
\item[1)]$E_0=E$;
\item[2)] For any natural number $i\geq0$, $E_{i+1}$ is the $\sO_X$-coherent subsheaf of $j_*\sE$ generated over $\sO_X$ by $E_i$ and $\nabla_\xi E_i$ for any local section $\xi$ of $\Theta_X(\log D)$.\end{itemize}
We also denote it by $E_i+\Theta_X(\log D)E_i$. Here $\Theta_X(\log D)E_i$ is understood as the $\sO_X$-coherent subsheaf of $j_*\sE$ spanned by the $\nabla_\xi E_i$.

By definition one has, for any divisor $\Delta$ with support in $D$,
\begin{itemize}
\item[1)] $E_i(\Delta)\subset j_*\sE$ and $\sE=j^*(E_i(\Delta))$ for each natural number $i\ge 0$;
\item[2)] The connection $j_*\nabla: j_*\sE\to \Omega^1_X(\log D)\otimes_{\sO_X} j_*\sE$ restricts for each natural number $i\ge 0$ to $\nabla: E_i(\Delta)\to \Omega^1_X(\log D)\otimes_{\sO_X} E_{i+1}(\Delta)$ defining the complex
$$\DR_{\log D}E_\bullet(\Delta):= \Omega_X^\bullet(\log D)\otimes_{\sO_X} E_\bullet(\Delta).$$
\end{itemize}

If $(\sE, \nabla)$ is regular singular, then $E_i=E_0$ and is locally free for all $i\ge 0$ by \cite{Del70}, {\it loc.\,cit.}

For a closed point $x\in D$, we do on $\hat X_x$ the analogous construction:
\begin{itemize}
\item[1)]$\widehat{ E_{0,x} }= (\widehat{E_0})_x$;
\item[2)] For any natural number $i\geq0$, $\widehat{E_{i+1,x}}$ is the $\sO_{\hat X_x}$-coherent subsheaf of $\hat \sE_x$ generated over $\sO_{\hat X_x}$ by $\widehat{E_{i,x}}$ and $\nabla_\xi \widehat{E_{i,x}}$ for any local section $\xi$ of $\Theta_{\hat X_x}(\log \hat D_x)$.
\end{itemize}

\begin{lemma}\label{lem:Ei}
For any closed point $x\in X$, one has $\widehat{ E_{i,x}}= (\widehat{E_i})_x$ and the $\sO_X$-coherent sheaves $E_i$ are all locally free.
\end{lemma}

\begin{proof}
As $\sO_X\to \sO_{\hat X_x}$ is flat, the second part of Lemma~\ref{lem:Ei} follows from the first part and $\widehat{E_{i,x}}$ being locally free.
We argue by induction in $i\ge 0$.
We first assume $i=0$. Then the first part $\widehat{ E_{0,x}}= (\widehat{E_0})_x$ is by definition.
By Definition~\ref{defn:goodlattice} 2), $(\widehat{E_0})_x$ is locally free if and only if $(\widehat{E'_0})_x$ is. So we are reduced to the unramified case, in which case this is part of Definition~\ref{defn:goodlattice} 1).
We now assume $i +1\ge 1$ and the statement for $i$.
By Lemma~\ref{lem:basic}, Lemma~\ref{lem:Ei} is equivalent to $(\widehat{ E_i} )_x + \widehat{ (\Theta_X(\log D)E_i)}_x$ being locally free. On the other hand, by definition
$\widehat{ (\Theta_X(\log D)E_i)}_x =\widehat{ \Theta_X(\log D)}_x (\widehat{E_i})_x$. Thus one has
$\widehat {E_{i+1, x}}= (\widehat{E_{i+1}})_x$.
It remains to prove local freeness.
Again by Definition~\ref{defn:goodlattice} 2) it follows from local freeness in the case of the existence of an unramified good Deligne-Malgrange lattice. Then if $I_\varphi$ denotes the effective pole divisor of $\varphi$,
$$ \widehat{E_{i,x}}=\bigoplus_{\varphi\in\Phi} (L_\varphi (i\cdot I_\varphi) \otimes R_{\varphi}).$$
This finishes the proof.
\end{proof}

\begin{lemma}\label{lem:basic}
Let $R\subset R'$ be a flat extension of commutative rings. Let $A,B,C$ be $R$-modules with $A,B\subset C$. Let $A'=A\otimes_R R'$ and similarly for $B, C$ be the base changed modules. Then
$A',B'\subset C'$,
$A'\cap B'=(A\cap B)'\subset C'$ and
$A'+B'=(A+B)'\subset C'$.
\end{lemma}

\begin{notations}
We define $\sD_X(\log D) \subset \sD_X$ to be the sheaf of subalgebras spanned by $\sO_X$ and $\Theta_X(\log D)$ and set $$V^0\sE= \textstyle\bigcup_{i\in \N} E_i {}\subset (V^0\sE)(\Delta)= \textstyle\bigcup_{i\in \N} E_i(\Delta)\subset j_*\sE.$$
\end{notations}

\begin{facts} \label{lem:V0Ei} By definition, one has
$$ \sD_X(\log D)\cdot E_0 (\Delta) =(V^0\sE)(\Delta) , \ \ \ \sD_X(\log D) \cdot V^0\sE (\Delta) =V^0\sE (\Delta),$$ defining
$$\DR_{\log D}V^0 \sE (\Delta)=\Omega^\bullet_X(\log D)\otimes V^0\sE (\Delta)$$
and the embedding of complexes
$$ \DR_{\log D} E_\bullet (\Delta) \To{\alpha} \DR_{\log D}V^0 \sE (\Delta)\To{\beta} \DR(j_*\sE),$$
where $\DR(j_*\sE)=\Omega^\bullet_X\otimes_{\sO_X} j_*\sE =\DR_{\log D}(j_*\sE)=\Omega^\bullet_X(\log D)\otimes j_*\sE $.
The $\sD_X(\log D)$-module $V^0\sE(\Delta)$ is coherent.
Moreover, $(\sE, \nabla)$ is regular singular if and only if $V^0\sE(\Delta)=E_0(\Delta)$, if and only if
$V^0\sE(\Delta)$ is a $\sO_X$-coherent subsheaf of $j_*\sE$.
\end{facts}

The rest of the section is devoted to prove
\begin{thm} \label{thm:qis}
Both $\alpha$ and $\beta$ are quasi-isomorphisms.
\end{thm}
Clearly this immediately implies Theorem~\ref{thm:main}.

We first treat $\alpha$. To this aim we regard the left-hand side of $\alpha$ as the zeroth term of the increasing filtration of $\DR_{\log D}V^0\sE$ defined by
\[
F_q\DR_{\log D}V^0\sE (\Delta)=
\begin{cases}
\{0\ra E_q (\Delta)\ra\Omega_X^1(\log D)\otimes E_{q+1} (\Delta)\ra\cdots\}&\text{if }q\geq0,\\
\phantom{\{}0&\text{if }q\leq-1.
\end{cases}
\]
So Theorem~\ref{thm:qis} for $\alpha$ is equivalent to the following proposition.

\begin{proposition}\label{prop:DRlogEV0}
For every $q\geq1$, the graded complex $\gr_q^F\DR_{\log D}V^0\sE (\Delta)$ has $\sO_X$-linear differentials and is quasi-isomorphic to zero.
\end{proposition}

\begin{proof}
The $\sO_X$-linearity is trivial. Faithful flatness of $\sO_X\to \sO_{\hat X_x}$ again implies that the proposition is true if and only if the analogous proposition on $\hat X_x$ is true for all closed points $x\in D$.
Again by Definition~\ref{defn:goodlattice} one reduces the problem to the case of the existence of an unramified good Deligne-Malgrange lattice: setting $\hat{\Delta'_x}=\hat{h_x}^*\hat\Delta_x$ and defining $\hat{E'_{i,x'}}$ from $\hat{E'_{0,x'}}$ as in 2) before Lemma \ref{lem:Ei}, we have $\hat{E_i(\Delta)}_x=(\hat{E'_{i,x'}}(\hat{\Delta'_x}))^G$ (same argument as in the proof of Proposition \ref{prop:beta} below). Let us consider the decomposition of Definition~\ref{defn:goodlattice}~1). For the component of $\hat E_x$ corresponding to $\varphi=0$, the filtration $F_q$ is constant and the assertion is obvious. For a component $L_\varphi\otimes R_{\varphi}$ with $\varphi\neq0$, the statement is equivalent to the $\sO_X$-linear complex $$(\Omega^\bullet_{\hat X_x}(\log D)\otimes_{\sO_{\hat X_x}} \sO_{I_\varphi} (\hat{\Delta'_x}+(q+\bullet)I_\varphi) \otimes_{\sO_{\hat X_x}} R_{\varphi}, \wedge d\varphi)$$ being quasi-isomorphic to zero. Moreover, it is enough to show this assertion after a finite extension of the ground field, so we can assume that Assumption \ref{assumption} holds. We set $\varphi=u(x)x^{-m}$ with $u\in\sO_{\hat x}^\times$ (i.e., $u(0)\neq0$) and $m\in\N^\ell\setminus\{0\}$. If $m_i\neq0$, then for any $k\geq0$,
\[
x_i\partial_{x_i}:\sO_{I_\varphi} (\hat{\Delta'_x}+(q+k)I_\varphi) \otimes_{\sO_{\hat X_x}} R_{\varphi}\to\sO_{I_\varphi} (\hat{\Delta'_x}+(q+k+1)I_\varphi) \otimes_{\sO_{\hat X_x}} R_{\varphi}
\]
is an isomorphism, being nothing but the multiplication by $-m_i$. The assertion follows.
\end{proof}

We now treat $\beta$. The assertion is local so we fix a closed point $x\in D$. Theorem~\ref{thm:mochizuki} yields for a closed point $x\in D$
natural numbers $m$ depending on $i, x$ and the choice of
$\widehat{ h_x}$ such that $\widehat{ h_x}$ is a Kummer cover along $\hat D_{i,x}$ which ramifies with ramification indices $m$.

\begin{definition} \label{defn:mix}
Fixing $(i,x)$, the minimal $m$ is called the ramification index of $(\sE,\nabla)$ at~$x$ along $D_i$
and is denoted by $m_{i,x}$.
\end{definition}

\begin{lemma} \label{lem:alg}
Fixing a closed point $x\in D$, the Galois cover $\widehat{h_x}$ from Definition~\ref{defn:goodlattice} \textup{2)} is algebraizable in the following sense. There exists a smooth projective variety $Z$ defined over~$k$, a finite union $z\in Z$ of closed points together with a flat finite morphism
$g: Z\to X$, finite étale Galois of group $H$ outside of a strict normal crossings divisor $\sD$ containing $D$, such that $Z\setminus g^{-1}(\sD)\hto Z$ is a good compactification, and there is a factorization
$$ \hat g_x: \hat Z_z \To{\lambda_x} \hat{X'_{x'} }\To{\widehat{h_x}} \hat X_x,$$
where $g_x$ is étale on $X_x \setminus D_x$. In particular, $G$ is a quotient of $H$.
\end{lemma}

\begin{proof}
We apply \cite[Th.\,17]{Kaw81} (together with \cite[Lem.\,5.2.4]{Mat02} for the Galois property) to construct $g$. We just have to check the last property. To say that $g_x$ is étale on $X_x \setminus D_x$ is to say that $x$ does not lie on the closure of $\sD \setminus D$. By the proof of \cite[Th.\,17]{Kaw81}, this divisor is generic in a very ample linear system, thus in particular can be chosen to avoid any $0$-dimensional subscheme. This finishes the proof.
\end{proof}

We set $j': Z\setminus g^{-1}(\sD)\to Z$ for the closed embedding and $\Delta''=g^*(\Delta)$.
We denote by $\beta'$ the corresponding embedding of de Rham complexes
$$ \DR_{\log g^{-1}(\sD)} [(V^0g^*\sE) (\Delta'')] \To{\beta'} \DR( j'_*j^{\prime*}g^*\sE).$$
We use the notations
$$\beta_x=\beta\otimes_{\sO_X} \sO_{X_x}, \ \ \beta'_z=\beta'\otimes_{\sO_Z} \sO_{Z_z},$$
and similarly $(\DR_{\log D} (V^0\sE) (\Delta))_x$ etc.

\begin{proposition}\label{prop:beta}
If $\beta'_z$ is a quasi-isomorphism, then so is $\beta_x$.
\end{proposition}

\begin{proof}
The morphism $ \beta_x:(\DR_{\log D}V^0\sE)_x\to (\DR j_*\sE)_x$ factors as
\[
(\DR_{\log D}V^0\sE)_x\to (\DR_{\log D}\sE)_x\to (\DR j_*\sE)_x,
\]
and, since the second morphism is clearly an isomorphism, we are reduced to showing that $(\DR_{\log D}(V^0\sE) (\Delta))_x\to(\DR_{\log D}j_*\sE)_x$ is a quasi-isomorphism.
By Lemma~\ref{lem:alg}, and uniqueness in Theorem~\ref{thm:mochizuki},
the sheaf $F_0^K$ of $K=\ker(H\to G)$-invariants of the Deligne-Malgrange lattice $F_0$ of $(j'_*j^{\prime*}g^*\sE)_z$
has formal germ $\wh{(F_0)}_z{}^K=\wh{(F_0^K)}_{x'}$ equal to the Deligne-Malgrange lattice of $\widehat{h_x} ^*\hat\sE_x$ on $\hat{X'_{x'}}$. This implies that $\wh{(F_0^H)}_x=\wh{(F_0)}_z{}^H=(\wh{(F_0)}_z{}^K)^G$ is equal to $\wh{(E_0)}_x$. We deduce from this that the two lattices~$F_0^H$ and $E_{0,x}$ of $(j_*\sE)_x$ coincide.

We now prove inductively on $i$ that $F_i^H=E_{i,x}$. We assume it is true for some $i\ge 0$. Recall that $g_x$ is étale away from $D_x$. As
$g_x^*\Theta_{X_x}(\log D_x)= \Theta_{Z_z}(\log g_x^{-1}D_x)$, one
has $$(\Theta_{Z_z}(\log g_x^{-1}D_x) F_i)^H= \Theta_{X_x}(\log D_x) F_i^H= \Theta_{X_x}(\log D_x)E_{i,x}. $$
On the other hand,
$$E_{i+1,x}= F_i^H+ (\Theta_{Z_z}(\log g_x^{-1}D_x) F_i)^H= \bigl(F_i+ (\Theta_{Z_z}(\log g_x^{-1}D_x) F_i)\bigr)^H =F_{i+1}^H$$
as
$E_{i,x}$ and $\Theta_{X_x}(\log D_x)E_{i,x}$ are sheaves of $\Q$-vector spaces, thus any local section $${(v+w)}\in \bigl(F_i+ (\Theta_{Z_z}(\log g_x^{-1}D_x) F_i)\bigr)^H$$ with $v\in F_i$, $w\in
(\Theta_{Z_z}(\log g_x^{-1}D_x) F_i)$ can be written as
$$ \frac{1}{|H|} \sum_{h\in H} h\cdot v + \frac{1}{|H|} \sum_{h\in H} h\cdot w \in F_i^H+ (\Theta_{Z_z}(\log g_x^{-1}D_x) F_i)^H.$$
We conclude that
$$ ((V^0 g^*\sE)_z)^H= (V^0\sE)_x,$$
and similarly $((V^0 g^*\sE)(\Delta'')_z)^H= V^0\sE(\Delta)_x$,
from which again by the compatibility of the differential forms we conclude
$$( \DR_{\log \sD} (V^0g^*\sE)(\Delta''))_z{}^H = (\DR_{\log \sD} (V^0\sE) (\Delta))_x {}=\DR_{\log D_x} (V^0\sE)(\Delta)_x .$$
On the other hand, one trivially has
$$ \DR(g^*j_*\sE)_z{}^H=\DR(j_*\sE)_x.$$
This finishes the proof.
\end{proof}

We may now assume that $\sE_x$ has an unramified good Deligne-Malgrange lattice.
The sheaf $\sD_X$ of differential operators on $X$ is canonically endowed with the filtration $F_\bullet\sD_X$ by the order of differential operators. In particular, $F_p\sD_X=0$ for $p\leq-1$, $F_0\sD_X=\sO_X$ and $F_1\sD_X$ is generated by $\sO_X$ and $\Theta_X$.
We first mimic the definition of the \emph{filtration by the order of the poles} (shifted by~$\Delta$) introduced by Deligne \cite[Chap.\,6]{Del70}, and its relation with the stupid filtration of the logarithmic de Rham complex. We set, for $p\in\Z$,

\[
P^p j_*\sE=(F_{-p}\sD_X)\cdot ( V^0\sE)(D+\Delta)\subset j_*\sE,
\]
so that, in particular,
\[
P^pj_*\sE=0\;(p\geq1),\quad P^0j_*\sE=(V^0\sE)(D+\Delta),\quad P^{-1}j_*\sE=P^0 j_*\sE+\Theta_X\cdot P^0j_*\sE.
\]
The de Rham complex is then filtered as usual, for $p\in\Z$,
\[
P^p\DR( j_*\sE)=\{0\ra P^pj_*\sE\ra\Omega^1_X\otimes P^{p-1} j_* \sE\ra\cdots\ra\Omega_X^{\dim X}\otimes P^{p-\dim X} j_*\sE\ra0\}.
\]
By definition, the differentials on the graded complex $P^p \DR j_* \sE /P^{p+1} \DR j_*\sE$ are $\sO_X$\nobreakdash-lin\-ear.
On the other hand, one has the stupid filtration $\sigma^{\geq p}$ on the
logarithmic de Rham complex of $V^0:=(V^0\sE)(\Delta)$:
\[
\sigma^{\geq p}\DR_{\log D}V^0=\{0\ra\cdots\ra0\ra\Omega_X^p(\log D)\otimes V^0\ra\cdots\ra\Omega_X^{\dim X}(\log D)\otimes V^0\ra0\},
\]
for which the graded complexes are just sheaves in various degrees. We use the same notations for the localization at a closed point $_x$.
Theorem~\ref{thm:qis} for $\beta$ is then a consequence of the following more precise theorem.

\begin{thm} \label{thm:P}
If $\sE_x$ has an unramified good Deligne-Malgrange lattice, the inclusion of filtered complexes
\[
(\DR_{\log D}V^0_x,\sigma^{\geq\bullet}) \hto(\DR\sE_x,P^\bullet)
\]
is a filtered quasi-isomorphism.
\end{thm}

\begin{proof}
The differentials of the graded complexes $\gr^p_\sigma\DR_{\log D}V^0_x$ and $\gr^p_P\DR\sE_x$ ($p\in\Z$) being $\sO_X$-linear, again by Definition~\ref{defn:goodlattice} we are reduced to the formal case.
In the regular case, the statement follows from the formal version of \hbox{\cite[Prop.\,II.3.13]{Del70}} as $\tau$ is admissible thus the condition on the residues is fulfilled.
In the case where $\varphi\neq0$, if $|I_\varphi|=D,$ we have $V^0=\sE$ and $P^p\sE=\sE$ for $p\leq0$. The assertion amounts to proving that the natural morphism
\[
(\DR_{\log D}\sE,\sigma^{\geq\bullet})\hto(\DR\sE,\sigma^{\geq\bullet})
\]
is a filtered quasi-isomorphism, which is obvious since $\Omega_X^p(\log D)\otimes\sE=\Omega_X^p\otimes\sE$ for every $p\geq0$. On the other hand, if $|I_\varphi| \subsetneq D$, $\sE$ is a successive extension of rank-one connections and it is enough to prove the assertion for such connections. Each such connection can be written as an external product of two terms, one term satisfying the assumption above, the other one being regular. For such a rank-one term, the assertion follows by using the regular case and the case $D_\varphi=D$, both proved above, by arguing as in \cite[p.\,81]{Del70}.
\end{proof}

\section{Remarks } \label{sec:rmks}

\subsection{Dimension one and Deligne's theorem}

If $n=1$, in which case $X$ is necessarily the normal compactification of $U$, this concept has been developed by Deligne \cite[\S6]{Del70}. He shows over $k=\C$ the existence of pairs $E_0\subset E_1$ with 1), 2), 3). He proves that although those pairs are not unique, ${\rm dim}_{\C} (E_1/E_0)_x$ is independent of the choice for all closed points $x$ on $D$ and defines the irregularity divisor $\sum_{x\in D} {\rm dim}_{\C} (E_1/E_0)_x \cdot x$ of $(\sE,\nabla)$ \cite[Lem.\,6.21]{Del70}. It has been then deemed `pairs of good lattices' in \cite[\S3]{BE04} in order to define local Fourier transforms of connections, then in \cite[\S3.1]{BBDE04} to compute the periods of the local Fourier transforms. While Deligne in {\it loc.\ cit.}\ constructed the lattices using the existence of a cyclic vector, the construction in \cite{BE04}, \cite{BBDE04} {\it loc.\ cit.}\ and \cite[\S5]{BBE02} uses the Turrittin-Levelt decomposition.

\subsection{Non-negative shifts}
The proof of Theorem~\ref{thm:main} yields that for any natural number~$a$, the
embeddings of complexes
$$\Omega^\bullet_{X}\otimes_{\sO_X} E_{\bullet} \hto \Omega^\bullet_{X}\otimes_{\sO_X} E_{a+\bullet} \hto j_*(\Omega^\bullet_U\otimes_{\sO_U}\sE)$$
are quasi-isomorphisms.

\subsection{Boundedness}
We remark the following.
\begin{claim}
The ramification indices $m_{i,x}$ of Definition~\ref{defn:mix} are bounded. Equivalently there is a natural number $M$ such that $m_{i,x}$ divides $M$ for all $i,x$.
\end{claim}
\begin{proof}
As $ \big(U, (\sE,\nabla)\big)$ and $X$ are defined over a field $k_0$ of finite type over $k$, Definition~\ref{defn:mix} over $k_0$ yields ramification indices $m_{\iota,u}(0)$ say for all closed points
$u \in D_\iota$ over $k_0$. Then $D_\iota \otimes_{k_0}k $ and $u\otimes_{k_0} k$ might further split,
with $x, D_i$ being one component. Then $\widehat{h_x}$ over $k$ is the pull-back of $\widehat{h_u}$ over $k_0$ localized at $x$ and $D_i$. Thus the $m_{i,x}$ over $k$ are the same as $m_{u,\iota}$ over $k_0$ for $k$. We now choose a complex embedding $ k_0\hto \C$. Again the $D_\iota \otimes_{k_0} \C $ and $u\otimes_{k_0} \C$ might further split and with the same argument, we just have to show boundedness for the $m_{a,z}$ where $z, D_a$ is one component of $(u, D_\iota)\otimes_{k_0} \C$. By \cite[p.\,2827, bullet point]{Moc09}, for each complex point $z$ (denoted by $P$ in {\it loc.\ cit.}) there exists an analytic neighbourhood~$X_z$ (denoted by $X_P$ in {\it loc.\ cit.}) and a Kummer cover $ \varphi: X'_z\to X_z$ such that $\varphi^*(\sE,\nabla)$ is unramified.
Thus for all
points $z ' \in D\cap X_z$, $m_{z',a}$ divides $m_{\varphi, a}$, the ramification index of~$\varphi$ along $D_a\cap X_z$.
As $D$ is compact in the analytic topology, it is covered by finitely many such analytic open sets $X_z$. This finishes the proof.
\end{proof}

\section{The logarithmic characteristic variety}
In this section, we give the proof of Theorem \ref{thm:main2} by reducing it to Theorem \ref{th:comdiag} below. The sheaf $\sD_X(\log D)\subset\sD_X$ is endowed with the induced filtration $F_\bullet\sD_X(\log D)$ by the order of differential operators. Definitions and results similar to those recalled in Sections \ref{subsec:right} and \ref{subsec:left} hold when $D$ is a normal crossings divisor in $X$ and upon replacing $\sD_X$ with $\sD_X(\log D)$ and correspondingly $T^*X$ with $T^*X(\log D)$, $\omega_X$ with $\omega_X(\log D)=\omega_X(D)$. For a coherent $\sD_X(\log D)$-module~$M$, one obtains
$$\Car_{\log}M \ {\rm in} \ K_0(\sO_{T^*X(\log D)}) .$$
The relation between the approach of Section \ref{subsec:computing} and the logarithmic approach will be obtained by factoring the log zero-section embedding $i_{\log}:X\hto T^*X(\log D)$ as
\[
X\Hto{i}T^*X\To{p}T^*X(\log D).
\]
We now take $(X, (E_0,\ldots, E_n))$ to be a good model of $(U, (\sE,\nabla))$, see Definition \ref{dfn:verygoodmodel}.
In particular, $j_*\sE$ has no turning point along $D$. By definition, $E_i=F_i\sD_X(\log D)\cdot E_0$ ($i\geq0$).

It is convenient to set $E_{-1}=0$. On the other hand, we set $V^0\sE=\bigcup_iE_i$ and $V^{-1}\sE=(V^0\sE)(D)$. The formation of $E_0$, $E_i$, $V^0\sE$ and $V^{-1}\sE$ is compatible with the restriction to the formal neighbourhood~$\hat X_x$ of a closed point $x\in X$.

Let us endow the $\sD_X(\log D)$-module $V^{-1}\sE:=(V^0\sE)(D)$ with the filtration $0=E_{-1}\subset E_0(D)\subset E_1(D)\subset\cdots$, so that its log de~Rham complex $\DR_{\log D}(V^{-1}\sE)$ is filtered by the formula
\[
F_q\DR_{\log D}(V^{-1}\sE)=\Bigl\{0\to E_q(D)\to\Omega^1_X(\log D)\otimes E_{q+1}(D)\to\cdots\Bigr\}
\]
\emph{for every $q\in\Z$}. By Proposition \ref{prop:DRlogEV0}, $\gr_q^F\DR_{\log D}(V^{-1}\sE)$ is acyclic for $q\geq1$ hence, by mimicking in the logarithmic case the argument leading to \eqref{eq:carM}, we find
\begin{equation}\label{eq:Licarlog}
\begin{split}
Li_{\log}^*\Car_{\log}(V^{-1}\sE)&=\sum_{i=0}^n(-1)^{n-i}\Bigl[\omega_X(\log D)^{-1}\otimes\Omega_X^i(\log D)\otimes E_i(D)\Bigr]\\
&=\sum_{i=0}^n(-1)^{n-i}\Bigl[\omega_X^{-1}\otimes\Omega_X^i(\log D)\otimes E_i\Bigr].
\end{split}
\end{equation}
Theorem \ref{thm:main2} immediately follows from the comparison between the characteristic class and the pullback of the log-characteristic class in the Grothendieck group $K_0(\sO_{T^*X})$.

\begin{thm}\label{th:comdiag}
We have the equality in $K_0(\sO_{T^*X})$:
\[
\Car(j_*\sE)=Lp^*\Car_{\log}(V^{-1}\sE).
\]
\end{thm}
In order to prove the theorem, one is led to compare $j_*\sE$ with $\sD_X\otimes_{\sD_X(\log D)}V^{-1}\sE$ and to extend this comparison to the graded modules with respect to a coherent $F$-filtration.

\subsection{Base change of the \texorpdfstring{$\sD_X(\log D)$}{D}-module \texorpdfstring{$V^{-1}\sE$}{VE}}

The following proposition is proved in \cite[\S\S5.3.2--5.3.3]{Moc15} in a more general context over the field $\C$.

\begin{proposition}\label{prop:isoDDlog}
The natural morphism $\sD_X\otimes_{\sD_X(\log D)}V^{-1}\sE\to\sE$ is an isomorphism.
\end{proposition}
\begin{proof}
The assertion is local formal, and is compatible with base change after a finite extension of the ground field, so we can assume that Assumption \ref{assumption} holds. We can also assume that the unramified good Deligne-Malgrange lattice $E_0$ comes from (see Definition \ref{defn:goodlattice}~2)) has only one component $L_\varphi\otimes R_\varphi$. We keep the notation $D$ for $\hat D_x$. We use the following notation:

\begin{itemize}
\item
Let $\hat{h_x}:\hat{X'_{x'}}\to\hat X_x$ be a finite morphism ramified along $D$ with Galois group $G$, such that $L_\varphi\otimes R_\varphi=\hat{\sE'_{x'}}:=\hat{h_x}^*\hat\sE_x$. Let $(x_1,\dots,x_n)$ and $(x'_1,\dots,x'_n)$ be local coordinates adapted to $\hat{h_x}$ and $D$, so that $D=(x_1\cdots x_\ell)$ and $\hat{h_x}^*(x_i)=x_i^{\prime\rho_i}$ ($i=1,\dots,\ell$). Denoting by $E'_0$ the Deligne-Malgrange lattice of $\hat{\sE'_{x'}}$, we have $E_0=(E'_0)^G$. It follows that $E_p=(E'_p)^G$ and thus $V^0\hat\sE_x=(V^0\hat{\sE'_{x'}})^G$. We then have $V^{-1}\hat\sE_x=(V^0\hat{\sE'_{x'}}(\hat{h_x}^*D))^G$. Moreover, with this identification, the action of $x_i\partial_{x_i}$ is induced by that of $(1/\rho_i)x'_i\partial_{x'_i}$ \hbox{($i=1,\dots,\ell$)}.

\item
We identify $D$ with the support of its pull-back by $\hat{h_x}$ and we decompose it as $D=D_1\cup D_2$, where $D_2$ supports the pole divisor of $\varphi$.
\item
We set $x^{\prime\rho}=x_1^{-\prime\rho_1}\cdots x_\ell^{-\prime\rho_\ell}$ and assume that $D_1=(x_1\cdots x_k)$ with $k\leq\ell$.
\item
We set $\hat \sD_x= \sO_{\hat X_x}\otimes_{\sO_X} \sD_X$.
\end{itemize}

We will prove the following two assertions.
\begin{enumerate}
\renewcommand{\theenumi}{\alph{enumi}}
\item\label{a}
$V^{-1}\hat\sE_x=(V^{-1}\hat\sE_x)(*D_2)$, so that $\hat\sE_x=(V^{-1}\hat\sE_x)(*D_1)$,
\item\label{b}
$\hat\sD_x\otimes_{\hat\sD_x(\log D_1)}V^{-1}\hat\sE_x\simeq (V^{-1}\hat\sE_x)(*D_1)$.
\end{enumerate}

Let us first check that these assertions imply the proposition. We first claim that
\begin{equation}\label{eq:Eiso}
\hat\sD_x(\log D_1)\otimes_{\hat\sD_x(\log D)}V^{-1}\hat\sE_x\to V^{-1}\hat\sE_x \quad\text{is an isomorphism.}
\end{equation}
Indeed, the composed natural morphism
\begin{equation}\label{eq:E''}
V^{-1}\hat\sE_x\to\hat\sD_x(\log D_1)\otimes_{\hat\sD_x(\log D)}V^{-1}\hat\sE_x\to V^{-1}\hat\sE_x
\end{equation}
is an isomorphism, hence the first morphism is injective. Let us check it is onto. Set $D_2=\{x_{k+1}\cdots x_\ell=0\}$. For any $P\in \hat\sD_x(\log D_1)$, there exists a sufficiently large integer $N$ such that $P\cdot\nobreak (x_{k+1}\cdots x_\ell)^N\in \hat\sD_x(\log D)$. Since any section $P\otimes m$ of $\hat\sD_x(\log D_1)\otimes_{\hat\sD_x(\log D)}V^{-1}\hat\sE_x$ can also be written as $P\cdot (x_{k+1}\cdots x_\ell)^N\otimes (x_{k+1}\cdots x_\ell)^{-N}m$ because of \eqref{a}, it is equal to
\[
1\otimes P\cdot (x_{k+1}\cdots x_\ell)^N\bigl[(x_{k+1}\cdots x_\ell)^{-N}m\bigr]=1\otimes Pm\in1\otimes V^{-1}\hat\sE_x,
\]
proving the surjectivity. The first morphism in \eqref{eq:E''} is thus bijective, and so is the second one. We conclude
\begin{align*}
\hat\sD_x\otimes_{\hat\sD_x(\log D)}V^{-1}\hat\sE_x&\simeq\hat\sD_x\otimes_{\hat\sD_x(\log D_1)}\bigl(\hat\sD_x(\log D_1)\otimes_{\hat\sD_x(\log D)}V^{-1}\hat\sE_x\bigr)\\
&\stackrel{\sim}\to\hat\sD_x\otimes_{\hat\sD_x(\log D_1)}V^{-1}\hat\sE_x\quad\text{by \eqref{eq:Eiso}}\\
&\simeq (V^{-1}\hat\sE_x)(*D_1)\simeq\hat\sE_x\quad\text{by \eqref{a}}.
\end{align*}

Let us now prove \eqref{a} and \eqref{b}. For \eqref{b}, we apply Proposition \ref{prop:generalV} to $D_1$ and $V^{-1}\hat\sE_x$ up to side-changing. In the left setting, the operators to be considered in loc.\,cit.\ are (up to sign) $\Euler_i-j$ for $j\geq0$. The properties \ref{prop:generalV}\eqref{prop:generalV1} and \eqref{prop:generalV2} read as follows for $V^0\hat{\sE'_{x'}}(\hat{h_x}^*D)$: For any subset $I\subset\{1,\dots,k\}$, the operators
\begin{itemize}
\item
$x'_i$ ($i\notin I$), ,
\item
$\Euler'_i-\rho_ij$ ($i\in I$, $j\geq0$)
\end{itemize}
act in an invertible way on $(x^{\prime-\rho})V^0(L_\varphi\otimes R_\varphi)/(x'_{i})_{i\in I}(x^{\prime-\rho})V^0(L_\varphi\otimes R_\varphi)$. These properties are easily checked, showing thereby, after taking $G$-invariants, that \eqref{b} holds.

Similarly, it is enough to prove \eqref{a} for $(x^{\prime-\rho})V^0(L_\varphi\otimes R_\varphi)$, for which the assertion is also easy.
\end{proof}

\subsection{Proof of Theorem \ref{th:comdiag}}
We use the notation and the terminology explained in the appendix, Section \ref{subsec:Rees}.

Recall that we denote by $F_\bullet \sD_X(\log D)$ (resp.\ $F_\bullet\sD_X$) the increasing filtration by the order of differential operators. We have in both cases $F_{-1}=0$ and $F_0=\sO_X$. Moreover, when regarding $\sD_X$ as a (left or right) $\sD_X(\log D)$-module, the filtration $F_\bullet\sD_X$ is an $F\sD_X(\log D)$-filtration of $\sD_X$.

Recall that $(E_i(D))_i$ is a coherent filtration of $V^{-1}\sE$ with respect to $F_\bullet \sD_X(\log D)$. The $\sD_X$-module $\sD_X\otimes_{\sD_X(\log D)}V^{-1}\sE$ is then endowed with a natural coherent filtration, namely
\begin{equation}\label{eq:FDVE}
F_i(\sD_X\otimes_{\sD_X(\log D)}V^{-1}\sE)=\sum_{j+k\leq i}\image\Bigl[F_j\sD_X\otimes_{\sO_X}E_k(D)\rightarrow\sD_X\otimes_{\sD_X(\log D)}V^{-1}\sE\Bigr].
\end{equation}
The $\sD_X$-linear isomorphism $\sD_X\otimes_{\sD_X(\log D)}V^{-1}\sE\to\sE$ (Proposition \ref{prop:isoDDlog}) enables us to transport this filtration as a coherent filtration $(F_i\sE)$ of the $\sD_X$\nobreakdash-mod\-ule~$\sE$, by defining
\begin{equation}\label{eq:FsE}
F_i\sE:=\sum_{j+k\leq i}F_j\sD_X\cdot E_k(D),
\end{equation}
where the sum is taken in $\sE$. We simply denote by $\gr (V^{-1}\sE)$ and $\gr\sE$ the graded modules with respect to these coherent filtrations $E_\bullet(D)$ and $F_\bullet\sE$ respectively. By definition we have
\begin{equation}\label{eq:grEgrtensor}
\gr\sE\simeq\gr^F(\sD_X\otimes_{\sD_X(\log D)}V^{-1}\sE).
\end{equation}
The Rees module $\wsE:=\bigoplus_iF_i\sE\cdot z^i\subset\sE\otimes_kk[z]$ with respect to the filtration \eqref{eq:FsE} is a left $\wsD_X$-module (see Section~\ref{subsec:Rees}). Then
\[
Li_z^*\wsE=i_z^*\wsE=\wsE/z\wsE=\gr\sE,
\]
where we regard $\gr\sE$ as a graded $\wsD_X$-module on which $z$ acts by zero, that is, a $\gr^F\sD_X$-module.

Similarly, by using the filtration $(E_i(D))_i$ of $V^{-1}\sE$, we define $\wVsE=\bigoplus_iE_i(D)\cdot z^i$, which is a left $\wsD_X(\log D)$-module. With this notation, \eqref{eq:FDVE} reads (see Lemma \ref{lem:tensorRees} up to changing the side):
\begin{equation}\label{eq:FDVERees}
(\sD_X\otimes_{\sD_X(\log D)}V^{-1}\sE)^\sim=\Bigl(\wsD_X\otimes_{\wsD_X(\log D)} \wVsE\Bigr)\Bigm/z\text{-torsion}.
\end{equation}
According to Lemma \ref{lem:tensorReesnonstrict} (up to side-changing), it is enough to prove the following.
\begin{proposition}\label{prop:isotilde}
We have
\[
\wsD_X\otimes^L_{\wsD_X(\log D)}\wVsE\simeq\wsD_X\otimes_{\wsD_X(\log D)}\wVsE.
\]
\end{proposition}

\begin{proof}
Since the statement is local, it can be proved after restricting to the formal neighbourhood of any point $x\in D$. Moreover, the formation of $V^{-1}$ and that of $\widetilde{V^{-1}}$ commute with tensoring with the ring $\sO_{\hat X_x}$ since $\hat{E_{0,x}}=\hat{E_0}_x$ (see 1) before Lemma \ref{lem:Ei}). Therefore, in the remaining part of the proof, \emph{we will assume that~$X$ is the formal neighbourhood of $x\in D$, but we will not change the notation for the sake of simplicity}. As both sides of the equation are compatible with a finite base field extension, we can suppose that Assumption \ref{assumption} holds. The proof of the proposition relies on the first part of Proposition \ref{prop:generalgr} of the appendix, after changing the side. It is enough to prove that any subsequence $(x_i)_{i\in I}$ of $(x_1,\dots,x_\ell)$ is a regular sequence for $\wVsE$. Moreover, since $\wVsE=\widetilde{V^0\sE}(D)$, $(x_i)_{i\in I}$ is a regular sequence for $\wVsE$ if and only if it is so for $\widetilde{V^0\sE}$. We thus argue with the latter module.

\subsubsection*{Reduction to the unramified case}
As in the proof of Proposition \ref{prop:isoDDlog}, and with similar notation, we find that $\widetilde{V^0\sE}=(\widetilde{V^0\sE'})^G$, and we can assume that $\sE'\simeq L_\varphi\otimes R_\varphi$. If we know that $(x^{\prime \rho_i}_i)_{i\in I}$ is a regular sequence for $\widetilde{V^0\sE'}$, then $(\hat{h_x}^*(x_i))_{i\in I}$ is a regular sequence, and by taking $G$-invariants we conclude that $(x_i)_{i\in I}$ is a regular sequence for $\widetilde{V^0\sE}$, as wanted.

\subsubsection*{The unramified case}
We now assume that $\sE$ is unramified. Let $\varphi(x)=u(x)/x^m$, where $m=(m_1,\dots,m_\ell,0,\dots,0)$, with $m_i\geq1$ for $i=k+1,\dots,\ell$ and $u(x)\in\sO_X$ with $u(0)\neq0$. Set $\sL_\varphi=(\sO_X(*D),d+d\varphi)$ and $D_m=\sum_{i=1}^\ell m_iD_i$. Then $L_{\varphi,p}=\sO_X(D_{pm})$ ($p\geq0$) and $E_p(D)=(V^0\sR)(D_{pm})$, so that, forgetting the connection, we have
\begin{equation}\label{eq:unramif}
\widetilde{V^0\sE}=\Bigl(\bigoplus_{p\geq0}(V^0R_\varphi)(D_{pm})\cdot z^p\Bigr)\subset R_\varphi[z].
\end{equation}
Since $\widetilde{V^0\sE}$ is graded, $(x_i^{\rho_i})_{i\in I}$ is a regular sequence for $\widetilde{V^0\sE}$ if and only if it is so for each graded piece $(V^0R_\varphi)(D_{pm})$. This holds since $(x_i^{\rho_i})_{i\in I}$ is a regular sequence for $V^0R_\varphi$.
\end{proof}

\newpage
\appendix
\begin{center}
\MakeUppercase{Appendix}\\[5pt]
\textsc{by Claude Sabbah}
\end{center}

\section{A reminder on characteristic varieties}\label{app:A}
In this section we reproduce \cite[\S A5]{MHM}. Recall that $n=\dim X$.

\subsection{Right \texorpdfstring{$\sD_X$}{DX}-modules}\label{subsec:right}
Recall that if $N$ is a right $\sD_X$-module, the Spencer complex of~$N$ is the complex
\[
\mathrm{Sp}(N):=\{0\ra N\otimes_{\sO_X}\wedge^n\Theta_X\To\delta\cdots\To\delta N\otimes_{\sO_X}\Theta_X\To\delta\underset\bullet N\ra0\},
\]
where the $_\bullet$ indicates the term in degree zero and where the differential $\delta$ is the $k$-linear map given for $m\in N$ by ($\hat\xi_i$ means omitting $\xi_i$ in the wedge product)
\begin{multline}\label{eq:delta}
m\otimes(\xi_1\wedge\cdots\wedge\xi_k)\Mto{\delta} \sum_{i=1}^{k}(-1)^{i-1}(m\xi_i)\otimes(\xi_1\wedge\cdots\wedge\widehat{\xi_i}\wedge\cdots\wedge\xi_k)\\[-3pt]
+\sum_{i<j}(-1)^{i+j}m\otimes\bigl([\xi_i,\xi_j]\wedge\xi_1\wedge\cdots\wedge\widehat{\xi_i}\wedge\cdots\wedge\widehat{\xi_j}\wedge\cdots\wedge\xi_k\bigr).
\end{multline}
Regarding $\sD_X$ as a right $\sD_X$-module, the Spencer complex $\mathrm{Sp}(\sD_X)$ is a complex in the category of left $\sD_X$-modules, by using the left $\sD_X$-module structure on $\sD_X$, and is a resolution of $\sO_X$ as a left $\sD_X$-module, by locally free $\sD_X$-modules. Moreover, there is a natural isomorphism
\begin{equation}\label{eq:SpencerN}
\mathrm{Sp}(N)\simeq N\otimes_{\sD_X}\mathrm{Sp}(\sD_X)=N\otimes_{\sD_X}^L\mathrm{Sp}(\sD_X).
\end{equation}
One concludes that
\begin{equation}\label{eq:SpencerNO}
N\otimes^L_{\sD_X}\sO_X\simeq N\otimes_{\sD_X}\mathrm{Sp}(\sD_X)\simeq\mathrm{Sp}(N).
\end{equation}
If $N$ is endowed with an $F\sD_X$-filtration, the formula for the differential $\delta$ shows that $\delta(F_pN\otimes\wedge^k\Theta_X)\subset F_{p+1}N\otimes\wedge^{k-1}\Theta_X$, and the Spencer complex is filtered by the formula
\begin{equation}\label{eq:SpencerFN}
F_p\mathrm{Sp}(N)=\{0\ra F_{p-n}N\otimes_{\sO_X}\wedge^n\Theta_X\To\delta\cdots\To\delta F_{p-1}N\otimes_{\sO_X}\Theta_X\To\delta F_pN\ra0\}.
\end{equation}
The graded complex $\gr^F\mathrm{Sp}(N)$ is thus expressed as
\begin{multline}\label{eq:SpencergrFN}
\gr^F\mathrm{Sp}(N)\\[-5pt]
=\{0\ra \gr^FN\otimes_{\sO_X}\wedge^n\Theta_X\To{\gr_1\delta}\cdots\To{\gr_1\delta} \gr^FN\otimes_{\sO_X}\Theta_X\To{\gr_1\delta} \gr^FN\ra0\},
\end{multline}
where $\gr_1\delta$ is defined by the first part of \eqref{eq:delta}, that is,
\begin{equation}
m\otimes(\xi_1\wedge\cdots\wedge\xi_k)\Mto{\gr_1\delta} \sum_{i=1}^{k}(-1)^{i-1}(\xi_i m)\otimes(\xi_1\wedge\cdots\wedge\widehat{\xi_i}\wedge\cdots\wedge\xi_k),
\end{equation}
where $\xi_i$ is now regarded as a linear form on $T^*X$.

We now explain that \eqref{eq:SpencerN} is compatible with taking $\gr^F$. Setting $\sO_{T^*X}:=\mathrm{Sym}\Theta_X$, one similarly regards $\gr^F\mathrm{Sp}(\sD_X)$ as a resolution of $\sO_X$ as an $\sO_{T^*X}$-module: one has
\[
\gr^F\mathrm{Sp}(\sD_X)=\{0\ra \sO_{T^*X}\otimes_{\sO_X}\wedge^n\Theta_X\To{\gr_1\delta}\cdots\To{\gr_1\delta} \sO_{T^*X}\otimes_{\sO_X}\Theta_X\To{\gr_1\delta} \sO_{T^*X}\ra0\}.
\]
Grading \eqref{eq:SpencerN} gives
\begin{equation}
\gr^F\mathrm{Sp}(N)\simeq \gr^FN\otimes_{\sO_{T^*X}}\gr^F\mathrm{Sp}(\sD_X)=\gr^FN\otimes^L_{\sO_{T^*X}}\gr^F\mathrm{Sp}(\sD_X),
\end{equation}
since each term of the complex $\gr^F\mathrm{Sp}(\sD_X)$ is $\sO_{T^*X}$-locally free. The graded analogue of~\eqref{eq:SpencerNO} is now
\begin{equation}\label{eq:SpencerNOgr}
\gr^FN\otimes^L_{\sO_{T^*X}}\sO_X\simeq\Car N\otimes^L_{\sO_{T^*X}}\gr^F\mathrm{Sp}(\sD_X) \simeq\gr^F\mathrm{Sp}(N)
\end{equation}
in the bounded derived category of $\sO_X$-modules. In the previous formulas, one forgets the information given by the grading (e.g.\ $\gr^FN=\bigoplus_pF_pN/F_{p-1}N$), as it is not to be used. In $K_0(\sO_X)$, \eqref{eq:SpencerNOgr} reads
\begin{equation}
Li^*\Car N=[\gr^F\mathrm{Sp}(N)].
\end{equation}

\subsection{Left \texorpdfstring{$\sD_X$}{sD}-modules}\label{subsec:left}
If $M$ is a left $\sD_X$-module, its de~Rham complex is defined as
\begin{equation}
\DR M=\{0\ra \underset\bullet M\To\nabla\Omega^1_X\otimes M\to\cdots\to\Omega^n_X\otimes M\ra0\}.
\end{equation}
Recall the $_\bullet$ indicated the degree zero term of the complex.
If $M$ is endowed with a coherent filtration $F_\bullet M$, we set
\begin{equation}
F_p\DR M=\{0\ra F_pM\To\nabla\Omega^1_X\otimes F_{p+1}M\to\cdots\to\Omega^n_X\otimes F_{p+n}M\ra0\}
\end{equation}
and we have
\begin{equation}
\gr^F\DR M=\{0\ra \gr^FM\To{\gr_1\nabla}\Omega^1_X\otimes \gr^FM\to\cdots\to\Omega^n_X\otimes \gr^FM\ra0\},
\end{equation}
which is a bounded complex in the category of $\sO_X$-modules.

The side-changing functor is defined by $N=\omega_X\otimes_{\sO_X}M$. Then $N$ is a right $\sD_X$\nobreakdash-module, and we have natural identifications (see e.g.\ \cite[Ex.\,A.5.9]{MHM})
\begin{equation}\label{eq:SpDR}
\mathrm{Sp}(N)\simeq\DR(M)[n], \quad\gr^F\mathrm{Sp}(N)\simeq\gr^F\DR(M)[n].
\end{equation}
Now, \eqref{eq:SpencerNOgr} gives
\begin{equation}
\gr^F(\omega_X\otimes_{\sO_X}M)\otimes^L_{\sO_{T^*X}}\sO_X \simeq\gr^F\mathrm{Sp}(\omega_X\otimes_{\sO_X}M)\simeq\gr^F\DR(M)[n]
\end{equation}
It follows that
\begin{equation}
\gr^FM\otimes^L_{\sO_{T^*X}}\sO_X\simeq\omega_X^{-1}\otimes\gr^F\DR(M)[n],
\end{equation}
and therefore, in $K_0(\sO_X)$,
\begin{equation}\label{eq:LiCarM}
Li^*\Car M=(-1)^n[\omega_X^{-1}\otimes\gr^F\DR(M)].
\end{equation}

\subsection{Computing \texorpdfstring{$Li^*\Car M$}{LiCarM}}\label{subsec:computing}
Let $M$ (resp.\ $N$) be a coherent left (resp.\ right) $\sD_X$-module and let $F_\bullet M$ (resp.\ $F_\bullet N$) be a coherent filtration. It is well known that there exists $p_0$ such that
\begin{equation}\label{eq:acyclicity}
\text{$\gr_p^F\DR M$ (resp.\ $\gr_p^F\mathrm{Sp}(N)$) is acyclic for any $p> p_0$.}
\end{equation}
Indeed, one first proves this for $N=\sD_X$, for which one knows that $\gr_p^F\mathrm{Sp}(\sD_X)$ is acyclic for any $p\geq1$ (see e.g.\ \cite[Ex.\,A.5.4(3)]{MHM}) and then for any coherent $(N,F_\bullet N)$ by using a suitable resolution of it by right filtered $\sD_X$-modules of the form $L\otimes_{\sO_X}\sD_X$, where $L$ is $\sO_X$-coherent. One deduces the lemma for $(M,F_\bullet M)$ by using the side-changing formulas \eqref{eq:SpDR}.

One sets $\gr_{\leq p}^F M:=\bigoplus_{q\leq p}\gr_q^FM$ and
\begin{multline}
\gr_{\leq p}^F\DR M\\[-3pt]
:=\{0\ra\gr_{\leq p}^F M\To{\gr_1\nabla}\Omega_X^1\otimes\gr_{\leq p+1}^F M\to\cdots\to\Omega_X^n\otimes\gr_{\leq p+n}^F M\ra0\}.
\end{multline}
The acyclicity property for $p>p_0$ implies that the inclusion of complexes
\[
\gr_{\leq p_0}^F\DR M\hto\gr^F\DR M
\]
is a quasi-isomorphism. Then \eqref{eq:LiCarM} reads
\begin{equation}\label{eq:LiCarMp0}
Li^*\Car M=(-1)^n[\omega_X^{-1}\otimes\gr^F_{\leq p_0}\DR(M)],
\end{equation}
that is
\begin{equation}\label{eq:carM}
\begin{split}
Li^*\Car M&=\sum_{i=1}^n(-1)^{n-i}[\omega_X^{-1}\otimes\Omega_X^i\otimes \gr^F_{\leq p_0+i}M]\\
&=\sum_{i=1}^n(-1)^{n-i}[\omega_X^{-1}\otimes\Omega_X^i\otimes F_{p_0+i}M],
\end{split}
\end{equation}
where the latter equality follows from $[\gr_{\leq p}^F M]=[F_pM]$ in $K_0(\sO_X)$.

However, it is in general difficult to determine the smallest $p_0$. One can nevertheless always assume that $F_{p_0+i}M=F_i\sD_X\cdot F_{p_0}M$ by taking $p_0$ large enough.

\section{ Defect of flatness of the differential operators on the~differential operators with logarithmic poles}\label{app:B}
In this section,
we work with right $\sD$-modules.

\subsection{Rees modules}\label{subsec:Rees}
The sheaf of rings $\sD_X$ is not flat over $\sD_X(\log D)$. In this appendix, we make explicit conditions on a $\sD_X(\log D)$-module $\sN$ so that, nevertheless, tensoring with~$\sN$ is exact. We enlarge the point of view to \emph{filtered} $\sD_X(\log D)$-modules by considering the associated Rees modules, in order to control the graded modules.

Recall that, to any object $M$ of an abelian category of (sheaves of) $k$-vector spaces endowed with an increasing filtration $F_\bullet M$, we associate its Rees module by introducing a new variable $z$ and by defining $R_FM$ as $\bigoplus_pF_pMz^p\subset k[z,z^{-1}]\otimes_k M$. Then, since the filtration in increasing, $R_FM$ is a $k[z]$-module. We apply this construction to $(\sD_X,F_\bullet\sD_X)$ and get
\[
\wsD_X:=\bigoplus_{p \in \N} F_p\sD_X\cdot z^p\subset\sD_X[z]:=\sD_X\otimes_kk[z],
\]
that we regard as a \emph{graded ring}, which is locally free over the graded ring $\wsO_X:=\sO_X[z]$ and generated as an $\wsO_X$-algebra by $\wsO_X$ and $\wTheta_X:=\Theta_X\otimes_kzk[z]$. In particular, $z$ is a central element in $\wsD_X$. Starting with $\sD_X(\log D)$, we define similarly $\wsD_X(\log D)$ and $\wTheta_X(\log D)$. We will work in the category $\Mod_\gr$ of \emph{graded modules} over these graded rings, and morphisms will be similarly graded of degree zero. We always assume that \emph{the grading is bounded from below}, that is, if $\wsN=\bigoplus_p(\wsN)_p$, we have $(\wsN)_p=0$ for $p\ll0$.

Let us note that
\[
k[z,z^{-1}]\otimes_{k[z]}\wsD_X=k[z,z^{-1}]\otimes_k\sD_X
\]
and a similar property for $\wsD_X(\log D)$.

Let $\wsN=\bigoplus_p(\wsN)_p$ be a graded $\wsD_X(\log D)$-module. Let $(\wsN)'_p$ denote the image of $(\wsN)_p$ in $\wsN[z^{-1}]:=k[z,z^{-1}]\otimes_{k[z]}\wsN$. Then $(\wsN)'_pz^{-p}\subset (\wsN)'_{p+1}z^{-p-1}$ and
\begin{equation}\label{eq:sN}
\sN:=\bigcup_p(\wsN)'_pz^{-p}\subset\wsN[z^{-1}]
\end{equation}
is a $\sD_X$-module, that we call the \emph{underlying $\sD_X$-module of $\wsN$}. We have
\[
\wsN[z^{-1}]=k[z,z^{-1}]\otimes_{k}\sN.
\]
We denote by $i_z^*$ the functor $\bullet\otimes_{k[z]}(k[z]/zk[z])$ and by $Li_z^*$ the associated derived functor. We have $i_z^*\wsD_X=\gr^F\sD_X$. We extend $i_z^*$ with the same notation as a functor $\Mod_\gr(\wsD_X)\to\Mod(\gr^F\sD_X)$ and $\catD^\rb_\gr(\wsD_X)\to\catD^\rb(\gr^F\sD_X)$. We also use the same notation when considering $\sD_X(\log D)$ instead of $\sD_X$.

We say that $\wsN$ is \emph{strict} if $z:\wsN\to\wsN$ is injective (equivalently, since $\wsN$ is graded, $\wsN$ is $k[z]$\nobreakdash-flat). A~graded $\wsD_X(\log D)$-module $\wsN$ is strict if and only if there exists an $F\sD_X(\log D)$-fil\-tration $F_\bullet\sN$ of the associated $\sD_X$-module $\sN$ defined by \eqref{eq:sN} such that $\wsN=\bigoplus_pF_p\sN z^p=:R_F\sN$ (the \emph{Rees module} of~$\sN$ with respect to $F_\bullet\sN$). Indeed, strictness is equivalent to the property that, for each~$p$,
\begin{equation}\label{eq:checkstrict}
z:(\wsN)_p\to(\wsN)_{p+1}\text{ is injective}.
\end{equation}
Setting $F_p\sN=(\wsN)'_pz^{-p}$ one has $\wsN=\bigoplus_pF_p\sN z^p$. Moreover, by strictness,
\[
Li_z^*\wsN=i_z^*\wsN=\gr^F\sN.
\]
A~similar characterization of strictness holds for graded $\wsD_X(\log D)$-modules.

\begin{lemma}\label{lem:tensorRees}
Assume $\wsN$ is a strict graded $\wsD_X(\log D)$-module $R_F\sN$. Then the quotient of the graded module $\wsN\otimes_{\wsD_X(\log D)}\wsD_X$ by its $z$-torsion is the Rees module of the filtration $F_\bullet(\sN\otimes_{\sD_X(\log D)}\sD_X)$ defined by
\[
F_p(\sN\otimes_{\sD_X(\log D)}\sD_X)=\image\Bigl[F_p(\sN\otimes_{\sO_X}\sD_X)\to(\sN\otimes_{\sD_X(\log D)}\sD_X)\Bigr],
\]
with
\[
F_p(\sN\otimes_{\sO_X}\sD_X)=\sum_q F_q\sN\otimes_{\sO_X}F_{p-q}\sD_X\subset\sN\otimes_{\sO_X}\sD_X.
\]
Moreover, we have
\begin{equation}\label{eq:tensorRees1}
Li_z^*\Bigl[(\wsN\otimes_{\wsD_X(\log D)}\wsD_X)\bigm/z\textup{-torsion}\Bigr]\simeq\gr^F(\sN\otimes_{\sD_X(\log D)}\sD_X)
\end{equation}
and
\begin{equation}\label{eq:tensorRees2}
Li_z^*(\wsN\otimes^L_{\wsD_X(\log D)}\wsD_X)\simeq\gr^F\sN\otimes^L_{\gr^F\sD_X(\log D)}\gr^F\sD_X.
\end{equation}
\end{lemma}

\begin{proof}
Since $\wsD_X$ is $\wsO_X$-locally free and since $\wsN$ is strict, $\wsN\otimes_{\wsO_X}\wsD_X$ is also strict, hence is the Rees module associated to a filtration $F_\bullet(\sN\otimes_{\sO_X}\sD_X)$. This is nothing but the filtration defined in the lemma.

There is a natural surjective composed morphism
\[
\wsN\otimes_{\wsO_X}\wsD_X\to\wsN\otimes_{\wsD_X(\log D)}\wsD_X\to(\wsN\otimes_{\wsD_X(\log D)}\wsD_X)/z\text{-torsion}.
\]
The last term is isomorphic to the Rees module of a filtration $F_\bullet(\sN\otimes_{\sD_X(\log D)}\sD_X)$ for which \eqref{eq:tensorRees1} holds and, by considering the component of degree $p$ for each $p$, one checks that this is the filtration defined in the lemma.

For \eqref{eq:tensorRees2}, we construct a graded resolution of $\wsD_X$ as a left graded $\wsD_X(\log D)$-module by flat graded $\wsD_X(\log D)$-modules~$L^\bullet$. Since $\wsD_X$ is $k[z]$-flat, one checks inductively that each differential $d_i:L^i\to L^{i-1}$ is strict, i.e., its cokernel is $k[z]$-flat. It follows that $i_z^*L^\bullet$ is a resolution of $i_z^*\wsD_X=\gr^F\sD_X$ by flat $\gr^F\sD_X(\log D)$-modules. We interpret $Li_z^*(\bullet)$ as $(k[z]/zk[z])\otimes^L_{k[z]}\bullet$. The associativity property as given e.g.\ in \cite[p.\,240]{Kashiwara03} leads to
\[
Li_z^*(\wsN\otimes^L_{\wsD_X(\log D)}\wsD_X)\simeq (Li_z^*\wsN)\otimes^L_{\wsD_X(\log D)}\wsD_X\simeq\gr^F\sN\otimes^L_{\wsD_X(\log D)}\wsD_X\simeq\gr^F\sN\otimes_{\wsD_X(\log D)}L^\bullet.
\]
Since $z$ acts by zero on $\gr^F\sN$, we have $\gr^F\sN\otimes_{\wsD_X(\log D)}L^\bullet\simeq\gr^F\sN\otimes_{\wsD_X(\log D)}i_z^*L^\bullet$, and the latter term is nothing but $\gr^F\sN\otimes_{\gr^F\sD_X(\log D)}i_z^*L^\bullet$. The choice of $L^\bullet$ implies that this is a realization of $\gr^F\sN\otimes^L_{\gr^F\sD_X(\log D)}\gr^F\sD_X$.
\end{proof}

As a consequence, we obtain the following criterion for a right graded $\sD_X(\log D)$-module~$\wsN$.

\begin{lemma}\label{lem:tensorReesstrict}
Assume that
\begin{enumerate}
\item\label{lem:tensorReesstrict1}
$\wsN$ is strict,
\item\label{lem:tensorReesstrict2}
$\wsN\otimes^L_{\wsD_X(\log D)}\wsD_X\simeq\wsN\otimes_{\wsD_X(\log D)}\wsD_X$,
\item\label{lem:tensorReesstrict3}
$\wsN\otimes_{\wsD_X(\log D)}\wsD_X$ is strict.
\end{enumerate}
Then we have
\[
\gr^F(\sN\otimes_{\sD_X(\log D)}\sD_X)\simeq\gr^F\sN\otimes_{\gr^F\sD_X(\log D)}\gr^F\sD_X\simeq\gr^F\sN\otimes^L_{\gr^F\sD_X(\log D)}\gr^F\sD_X.
\]
\end{lemma}

\begin{proof}
It is enough to prove the isomorphism between the first and the third term, since this would imply that the third term is isomorphic to its $\sH^0$, that is, the second term. We~write
\begin{equation}\label{eq:tensorReesstrict}
\begin{split}
\gr^F\sN\otimes^L_{\gr^F\sD_X(\log D)}&\gr^F\sD_X\\
&\simeq Li_z^*(\wsN\otimes^L_{\wsD_X(\log D)}\wsD_X)\quad\text{(by \eqref{eq:tensorRees2})}\\
&\simeq Li_z^*(\wsN\otimes_{\wsD_X(\log D)}\wsD_X)\quad\text{(by Assumption \eqref{lem:tensorReesstrict2})}\\
&\simeq Li_z^*\Bigl[(\wsN\otimes_{\wsD_X(\log D)}\wsD_X)/z\textup{-torsion}\Bigr]\quad\text{(by Assumption \eqref{lem:tensorReesstrict3})}\\
&\simeq \gr^F(\sN\otimes_{\sD_X(\log D)}\sD_X)\quad\text{(by \eqref{eq:tensorRees1})}.
\end{split}
\end{equation}
\end{proof}

We now relax Condition \ref{lem:tensorReesstrict}\eqref{lem:tensorReesstrict3}, but we assume that $\wsN$ is $\wsD_X(\log D)$-coherent. Let us denote by $\wsT$ the $z$-torsion of $\wsN\otimes_{\wsD_X(\log D)}\wsD_X$. It is a coherent graded $\wsD_X$-module, the sections of which are annihilated by some power of $z$. It has thus a finite filtration such that each graded piece $\gr_\ell\wsT$ is a coherent graded $\wsD_X$-module annihilated by $z$, hence a coherent graded $\gr^F\sD_X$-module. By the first lines of \eqref{eq:tensorReesstrict}, we find
\begin{align*}
\sH^0(\gr^F\sN\otimes^L_{\gr^F\sD_X(\log D)}\gr^F\sD_X)&=\gr^F\sN\otimes_{\gr^F\sD_X(\log D)}\gr^F\sD_X=i_z^*(\wsN\otimes_{\wsD_X(\log D)}\wsD_X)\\
\sH^{-1}(\gr^F\sN\otimes^L_{\gr^F\sD_X(\log D)}\gr^F\sD_X)&=\ker z,
\end{align*}
and $\sH^i(\gr^F\sN\otimes^L_{\gr^F\sD_X(\log D)}\gr^F\sD_X)=0$ for $i\neq0,-1$. We also have an exact sequence
\[
\wsT/z\wsT\to\gr^F\sN\otimes_{\gr^F\sD_X(\log D)}\gr^F\sD_X\to\gr^F(\sN\otimes_{\sD_X(\log D)}\sD_X)\to0,
\]
but one does not expect in general an isomorphism between $\gr^F\sN\otimes_{\gr^F\sD_X(\log D)}\gr^F\sD_X$ and $\gr^F(\sN\otimes_{\sD_X(\log D)}\sD_X)$.

\begin{lemma}\label{lem:tensorReesnonstrict}
Let $\wsN$ be a coherent graded $\wsD_X(\log D)$-module. Assume that Properties~\ref{lem:tensorReesstrict}\eqref{lem:tensorReesstrict1} and \ref{lem:tensorReesstrict}\eqref{lem:tensorReesstrict2} hold. Then we have the equality in $K_0(\sO_{T^*X})$:
\[
[\gr^F\sN\otimes^L_{\gr^F\sD_X(\log D)}\gr^F\sD_X]=[\gr^F(\sN\otimes_{\sD_X(\log D)}\sD_X)].
\]
\end{lemma}

\begin{proof}
If we do not assume \ref{lem:tensorReesstrict}\eqref{lem:tensorReesstrict3}, we can nevertheless write in $K_0(\sO_{T^*X})$, with the notation as above:
\[
[\gr^F\sN\otimes^L_{\gr^F\sD_X(\log D)}\gr^F\sD_X]=[\gr^F(\sN\otimes_{\sD_X(\log D)}\sD_X)]+\sum_\ell[Li_z^*\gr_\ell\wsT].
\]
On the other hand, we have
\[
[Li_z^*\gr_\ell\wsT]=[\gr_\ell\wsT]-[\gr_\ell\wsT]=0
\]
in $K_0(\sO_{T^*X})$.
\end{proof}

\subsection{Flatness properties}
\begin{assumption}\label{assumption}
We denote by $\hat X$ a formal neighbourhood of a closed point $x\in D$ and we assume that there exists a regular system of parameters $(x_1,\dots,x_n)$ in $\sO_{\hat X,x}$ such that $\hat D=\{x_1\cdots x_\ell=0\}$.

\emph{For the sake of simplicity, we still denote $\hat X$ by $X$ and $\hat D$ by $D$.}
\end{assumption}

It is straightforward to check that the results of Section \ref{subsec:Rees} apply in this setting. The following proposition and its proof are inspired by~\cite{Wei17}.

\begin{proposition}\label{prop:generalgr}
Let $\wsN$ be a right graded $\wsD_X(\log D)$-module. Assume that any subsequence of the sequence $(x_1,\dots,x_\ell)$ is a regular sequence for~$\wsN$. Then
\[
\wsN\otimes^L_{\wsD_X(\log D)}\wsD_X\simeq\wsN\otimes_{\wsD_X(\log D)}\wsD_X.
\]
If moreover every quotient $\wsN/\sum_{i\in I}\wsN x_i$ ($I\subset\{1,\dots,\ell\}$) is strict, then $\wsN\otimes_{\wsD_X(\log D)}\wsD_X$ is strict.
\end{proposition}

\begin{proof}
Recall that $\Spencer \wsD_X(\log D)$ is the complex having $\wsD_X(\log D)\otimes_{\wsO_X}\wedge^k\wTheta_X(\log D)$ as its term in degree $-k$, and differential the left $\wsD_X(\log D)$-linear morphism
\[
\wsD_X(\log D)\otimes_{\wsO_X}\wedge^k\wTheta_X(\log D)\To{\delta}\wsD_X(\log D)\otimes_{\wsO_X}\wedge^{k-1}\wTheta_X(\log D)
\]
given, for $\btheta=\theta_1\wedge\cdots\wedge\theta_k$
\[
\delta(P\otimes\btheta)=\sum_{i=1}^k(-1)^{i-1}(P\cdot\theta_i)\otimes\widehat{\btheta_i}+\sum_{i<j}(-1)^{i+j}P\otimes([\theta_i,\theta_j]\wedge\widehat{\btheta_{i,j}}),
\]
with $\widehat{\btheta_i}=\theta_1\wedge\cdots \wedge\theta_{i-1}\wedge\theta_{i+1}\wedge\cdots \wedge\theta_k$, and a similar meaning for $\widehat{\btheta_{i,j}}$. Since $\Sp(\wsD_X(\log D))$ is a resolution of $\wsO_X$ by locally free left $\wsD_X(\log D)$-modules which are thus $\wsO_X$-locally free, we have
\[
\wsN\simeq\wsN\otimes_{\wsO_X}\Spencer \wsD_X(\log D),
\]
with their right $\wsD_X(\log D)$-module structure, by using the tensor right structure on the right-hand side. The complex $\wsN\otimes_{\wsO_X}\Spencer \wsD_X(\log D)$ has
\[
\wsN\otimes_{\wsO_X}(\wsD_X(\log D)\otimes_{\wsO_X}\wedge^k\wTheta_X(\log D))
\]
as its term in degree $-k$, and differential $\id\otimes\delta$, which is right $\wsD_X(\log D)$-linear for the tensor right structure.

We recall that there are two natural structures of right $\wsD_X(\log D)$-module on the tensor product
\[
\wsN\otimes_{\wsO_X}\wsD_X(\log D)\otimes_{\wsO_X}\wedge^k\wTheta_X(\log D).
\]
The \emph{tensor structure} is obtained by using the right structure on $\wsN$ and the left structure on $\wsD_X(\log D)$. On the other hand, the \emph{trivial structure}, for which we rather denote the tensor product as
\[
(\wsN\otimes_{\wsO_X}\wedge^k\wTheta_X(\log D))\otimes_{\wsO_X}\wsD_X(\log D),
\]
is obtained by using the right structure on $\wsD_X(\log D)$ and by completely forgetting the right action of derivations on $\wsN$ while only remember its $\wsO_X$-structure. However, there exists a unique involution of right $\wsD_X$-modules
\[
\wsN\otimes_{\wsO_X}(\wsD_X(\log D)\otimes_{\wsO_X}\wedge^k\wTheta_X(\log D))\To{\sim}(\wsN\otimes_{\wsO_X}\wedge^k\wTheta_X(\log D))\otimes_{\wsO_X}\wsD_X(\log D)
\]
extending the natural involution of $\wsO_X$-modules
\[
\wsN\otimes_{\wsO_X}(\wsO_X\otimes_{\wsO_X}\wedge^k\wTheta_X(\log D))\To{\sim}(\wsN\otimes_{\wsO_X}\wedge^k\wTheta_X(\log D))\otimes_{\wsO_X}\wsO_X.
\]

Let us make explicit the differential $\delta$. For $P\in \wsD_X(\log D)$, the element $[n\otimes(1\otimes\btheta)]\cdot P$ (tensor structure) is complicated to express, but we must have, by right $\wsD_X(\log D)$-linearity of $\id\otimes\delta$,
\begin{align*}
(\id\otimes\delta)\bigl[(n\otimes(1\otimes&\btheta))\cdot P\bigr]=\bigl[(\id\otimes\delta)(n\otimes(1\otimes\btheta))\bigr]\cdot P\\
&=\biggl[n\otimes\Bigl[\sum_{i=1}^k(-1)^{i-1}\theta_i\otimes\widehat{\btheta_i}+\sum_{i<j}(-1)^{i+j}1\otimes([\theta_i,\theta_j]\wedge\widehat{\btheta_{i,j}})\Bigr]\biggr]\cdot P.
\end{align*}
We now write
\[
n\otimes(\theta_i\otimes\widehat{\btheta_i})=n\theta_i\otimes(1\otimes\widehat{\btheta_i})-[n\otimes(1\otimes\widehat{\btheta_i})]\cdot\theta_i,
\]
so the previous formula reads, after the involution transforming the tensor structure to the trivial one, by denoting $\delta_\triv$ the corresponding differential:
\begin{multline}\label{eq:deltatriv}
\delta_\triv\bigl[(n\otimes\btheta)\otimes P\bigr]=\sum_{i=1}^k(-1)^{i-1}(n\theta_i\otimes\widehat{\btheta_i})\otimes P\\[-8pt]
-\sum_{i=1}^k(-1)^{i-1}(n\otimes\widehat{\btheta_i})\otimes (\theta_iP)+\sum_{i<j}(-1)^{i+j}(n\otimes([\theta_i,\theta_j]\wedge\widehat{\btheta_{i,j}}))\otimes P\\[-8pt]
=\bigl[\delta_\wsN(n\otimes\btheta)\bigr]\otimes P-\sum_{i=1}^k(-1)^{i-1}(n\otimes\widehat{\btheta_i})\otimes (\theta_iP),
\end{multline}
where $\delta_\wsN$ is the differential of the Spencer complex $\Spencer_{\log}\wsN$ of $\wsN$ as a right $\wsD_X(\log D)$-module.

We obtain, due to the local $\wsO_X$-freeness of $\wsD_X(\log D)$ and $\wsD_X$,
\begin{align}\label{eq:resol}
\wsN\otimes^L_{\wsD_X(\log D)}\wsD_X&\simeq(\wsN\otimes_{\wsO_X}\Spencer \wsD_X(\log D))\otimes^L_{\wsD_X(\log D)}\wsD_X\notag\\
&\simeq\bigl((\wsN\otimes_{\wsO_X}\wedge^{-\bullet}\wTheta_X(\log D))\otimes_{\wsO_X}\wsD_X(\log D),\;\delta_\triv\bigr)\otimes_{\wsD_X(\log D)}^L\wsD_X\notag\\
&\simeq\bigl((\wsN\otimes_{\wsO_X}\wedge^{-\bullet}\wTheta_X(\log D))\otimes^L_{\wsO_X}\wsD_X(\log D),\;\delta_\triv\bigr)\otimes_{\wsD_X(\log D)}^L\wsD_X\\
&\simeq\bigl((\wsN\otimes_{\wsO_X}\wedge^{-\bullet}\wTheta_X(\log D))\otimes^L_{\wsO_X}\wsD_X,\;\delta_\triv\bigr)\notag\\
&\simeq\bigl((\wsN\otimes_{\wsO_X}\wedge^{-\bullet}\wTheta_X(\log D))\otimes_{\wsO_X}\wsD_X,\;\delta_\triv\bigr).\notag
\end{align}
In the last line, $\delta_\triv$ is given by \eqref{eq:deltatriv}, where $P$ is now a local section of $\wsD_X$. We have thus realized $\wsN\otimes^L_{\wsD_X(\log D)}\wsD_X$ as a complex $(\sF^\bullet\otimes_{\wsO_X}\wsD_X,\delta_\triv)$, where each term $\sF^k$ is an $\wsO_X$-module.

With respect to the filtration $\sF^\bullet\otimes_{\wsO_X}F_k\wsD_X$, $\delta_\triv$ has degree one, and the differential $\gr_1^F\delta_\triv$ of the graded complex $\sF^\bullet\otimes_{\wsO_X}\gr^F\wsD_X$ is expressed as
\[
\gr_1^F\delta_\triv\bigl[(n\otimes\btheta)\otimes Q\bigr]=\sum_{i=1}^k(-1)^i(n\otimes\widehat{\btheta_i})\otimes (\theta_i\cdot Q)
\]
for a local section $Q$ of $\gr^F\wsD_X$. The filtration $F_p(\sF^\bullet\otimes_{\wsO_X}\wsD_X,\delta_\triv)$ whose term in degree $-k$ is $\sF^{-k}\otimes_{\wsO_X}F_{p-k}\wsD_X$ satisfies $F_p(\sF^\bullet\otimes_{\wsO_X}\wsD_X,\delta_\triv)=0$ for $p<0$ and we have
\begin{equation}\label{eq:grFD}
\gr^F(\sF^\bullet\otimes_{\wsO_X}\wsD_X,\delta_\triv)=(\sF^\bullet\otimes_{\wsO_X}\gr^F\wsD_X,\gr_1^F\delta_\triv),
\end{equation}
compatible with the grading, if the grading on the right-hand side takes into account the cohomology degree.

\begin{lemma}
If the first condition of Proposition \ref{prop:generalgr} is fulfilled, the graded complex $(\sF^\bullet\otimes_{\wsO_X}\gr^F\wsD_X,\gr_1^F\delta_\triv)$ has zero cohomology in any degree $i\neq0$. If the second condition is also fulfilled, then the cohomology in degree zero is strict.
\end{lemma}

\begin{proof}
Let us consider the basis
\[
x_1\wpartial_{x_1},x_2\wpartial_{x_2},\dots,x_\ell\wpartial_{x_\ell},\wpartial_{x_{\ell+1}},\dots,\wpartial_{x_n}
\]
of $\wTheta_X(\log D)$, so that, by replacing $x_i\wpartial_{x_i}$ with $\wpartial_{x_i}$ we obtain a basis of $\wTheta_X$. Let $\wxi_1,\wxi_2,\dots,\wxi_n$ resp.\ $x_1\wxi_1,x_2\wxi_2,\dots,\wxi_n$ be the corresponding basis of $\gr_1^F\wsD_X$ resp.\ $\gr_1^F\wsD_X(\log D)$. Then $\gr^F(\sF^\bullet\otimes_{\wsO_X}\wsD_X,\delta_\triv)$ is identified with a Koszul complex. More precisely, it isomorphic to the simple complex associated to the $n$-cube with vertices
\[
\wsN\otimes_{\wsO_X}\gr^F\wsD_X=\wsN\otimes_{k}k[\wxi_1,\dots,\wxi_n]
\]
and arrows in the $i$-th direction all equal to multiplication by~$\wxi_i$ if $i>\ell$ and by $x_i\otimes\wxi_i$ if $i\leq\ell$. In such a way we obtain that $(\sF^\bullet\otimes_{\wsO_X}\gr^F\wsD_X,\gr_1^F\delta_\triv)$ is quasi-isomorphic to the simple complex attached to the $\ell$-cube having
\[
\wsN\otimes_{k}k[\wxi_1,\dots,\wxi_\ell]\To{x_i\otimes\wxi_i}\wsN\otimes_{k}k[\wxi_1,\dots,\wxi_\ell]
\]
as its arrow in the $i$-th directions.

Let us prove by induction on $\ell$ that, under the first assumption of Proposition \ref{prop:generalgr}, this complex has cohomology in degree zero only, and this cohomology is isomorphic to a direct sum of terms, each of which isomorphic to $\wsN/\sum_{i\in I}\wsN x_i$ for some $I\subset \{1,\dots,\ell\}$. This will give the first part of the lemma.

Indeed, since $x_1$ is injective on $\wsN$, each arrow $x_1\otimes\wxi_1$ is injective with cokernel isomorphic to $\wsN_1\otimes_kk[\wxi_2,\dots,\wxi_\ell]$, where $\wsN_1=\wsN\oplus\bigoplus_{k\geq1}(\wsN/\wsN x_1)[\wxi_1^k]$ and $[\wxi_1^k]$ is the image of $\wxi_1^k$ in the cokernel. It follows that the complex we consider is quasi-isomorphic to the simple complex attached to the $(\ell-1)$-cube with vertices $\wsN_1\otimes_k\nobreak k[\wxi_2,\dots,\wxi_\ell]$ and arrows $x_i\otimes\wxi_i$ ($i=2,\dots,\ell)$. Now, any subsequence of $(x_2,\dots,x_\ell)$ is a regular sequence for $\wsN_1$ by our assumption, and we obtain the desired assertion by induction on $\ell$.

For the second part of the lemma, we note that, by the second assumption of Proposition \ref{prop:generalgr}, each term $\wsN/\sum_{i\in I}\wsN x_i$ is strict, so the cohomology in degree zero is strict.
\end{proof}

\subsubsection*{End of the proof of Proposition \ref{prop:generalgr}}
By \eqref{eq:grFD}, the lemma applies to the graded complex $\gr^F(\sF^\bullet\otimes_{\wsO_X}\wsD_X,\delta_\triv)$ and therefore each $\gr^F_p(\sF^\bullet\otimes_{\wsO_X}\wsD_X,\delta_\triv)$ has cohomology in degree zero at most and, if the supplementary strictness assumption on the quotients of $\wsN$ is fulfilled, the cohomology in degree zero is strict. It follows that each $F_p(\sF^\bullet\otimes_{\wsO_X}\wsD_X,\delta_\triv)$ satisfies the same property since $F_{-1}(\sF^\bullet\otimes_{\wsO_X}\wsD_X,\delta_\triv)=0$. Passing to the inductive limit, we conclude that so does the complex \hbox{$(\sF^\bullet\otimes_{\wsO_X}\wsD_X,\delta_\triv)$}.
\end{proof}

\subsection{A criterion for having \texorpdfstring{$\sN\otimes_{\sD_X(\log D)}\sD_X\simeq\sN(*D)$}{crit}}\mbox{}

\emph{Assumption \ref{assumption} and the corresponding simplifying notation remain in order in this section.}

For $I\subset \{1,\dots,\ell\}$ we set $I^\rc=\{1,\dots,\ell\}\setminus I$, $D_I=\bigcap_{i\in I}D_i$ and, for each $j\in I^\rc$, $D_{I,j}=D_I\cap D_j$ and $D_I^{(I^\rc)}=\bigcup_{j\in\{1,\dots,\ell\}\setminus I}D_{I,j}$, so that $D_I^{(I^\rc)}$ is a divisor with normal crossings in $D_I$. If $I=\emptyset$, we have $D_I=X$ and $D_I^{(I^\rc)}=D$.

We identify $i_{D_I}^*\sO_X$ with $\sO_{D_I}$ and $i_{D_I}^*\sD_X(\log D)$ with $\sD_{D_I}(\log D_I^{(I^\rc)})[(\Euler_i)_{i\in I}]$, where $\Euler_i$ is the class of $x_i\partial_{x_i}$. It is a central element in this ring.

The following result is inspired by \cite[Lem.\,3.1.2]{Moc15}.

\begin{proposition}\label{prop:generalV}
Let $\sN$ be a right $\sD_X(\log D)$-module. Together with the assumption above, assume the following property: for any (possibly empty) subset $I$ of $\{1,\dots,\ell\}$,
\begin{enumerate}
\item\label{prop:generalV1}
$i_{D_I}^*\sN\subset(i_{D_I}^*\sN)(*D_I^{(I^\rc)})$,
\item\label{prop:generalV2}
for all $j\in I$ and all $k\geq1$,
\[
(\Euler_j+k):i_{D_I}^*\sN\to i_{D_I}^*\sN\quad\text{is bijective}.
\]
\end{enumerate}
Then the natural morphism $\sN\otimes_{\sD_X(\log D)}\sD_X\to\sN(*D)$ is an isomorphism.
\end{proposition}

In order to argue by induction on $\#D$, we first consider how the properties in Proposition~\ref{prop:generalV} are preserved when
\begin{itemize}
\item
$D=D_1\cup D'$ and $D$ is replaced with $D'$ (in the local setting, we assume that $D_1=\{x_1=0\}$ and $D'=\{x_2\cdots x_\ell=0\}$),
\item
correspondingly, $\sN$ is replaced with $\sN\otimes_{\sD_X(\log D)}\sD_X(\log D')$.
\end{itemize}

In order to simplify the formulas, we use in this section the simplified notation below:
\[
\sD_{X,\log}:=\sD_X(\log D),\quad \sD'_{X,\log}:=\sD_X(\log D'),\quad \sN'=\sN\otimes_{\sD_{X,\log}}\sD'_{X,\log}.
\]
Since $\sD_{X,\log}(*D_1)=\sD'_{X,\log}(*D_1)$, we have $\sN'(*D_1)\simeq\sN(*D_1)$, and we have a natural localization morphism $\sN'\to\sN'(*D_1)=\sN(*D_1)$.

\begin{lemma}\label{lem:tensorloc}
Let $\sN$ be a right $\sD_{X,\log}$-module such that $\sN\subset\sN(*D_1)$ and \ref{prop:generalV}\eqref{prop:generalV2} for $I=\{1\}$ holds. Then the localization morphism
\[
\sN'\to\sN'(*D_1)=\sN(*D_1)
\]
is an isomorphism.
\end{lemma}

\begin{proof}
The statement is local. For $k\geq0$, we set
\[
\sD'_{X,\log,\leq k}=\sum_{j=0}^k\partial_{x_1}^j\sD_{X,\log}=\sum_{j=0}^k\sD_{X,\log}\,\partial_{x_1}^j.
\]
We will prove by induction on $k$ that $\sN\otimes_{\sD_{X,\log}}\sD'_{X,\log,\leq k}$ injects into $\sN(*D_1)$ with image equal to $\sN x_1^{-k}$, the case $k=0$ being given by \ref{prop:generalV}\eqref{prop:generalV1} for $I=\emptyset$. The composition
\[
\sN\otimes_{\sD_{X,\log}}\sD'_{X,\log,\leq k-1}\to\sN\otimes_{\sD_{X,\log}}\sD'_{X,\log,\leq k}\to\sN(*D),
\]
being injective by induction, so is the first morphism, whose cokernel will be denoted by $\gr_k(\sN\otimes_{\sD_{X,\log}}\sD'_{X,\log})$. Since $\sN$ is acted on by $x_1\partial_{x_1}$, hence by $\partial_{x_1}^kx_1^k=\prod_{j=1}^k(x_1\partial_{x_1}+j)$, the second morphism factorizes through $\sN x_1^{-k}\subset\sN(*D_1)$: an element $m\otimes\partial_{x_1}^k$ has image $m\partial_{x_1}^k=(m\partial_{x_1}^kx_1^k)x_1^{-k}$, and $m\partial_{x_1}^kx_1^k$ is the image of $(m\partial_{x_1}^kx_1^k)\otimes1$, hence belongs to $\sN$. We consider the commutative diagram of exact sequences, where the first two terms of the lower line are regarded in $\sN(*D_1)$,
\[
\xymatrix@C=.5cm{
0\ar[r]&\sN\otimes_{\sD_{X,\log}}\sD'_{X,\log,\leq k-1}\ar[r]\ar[d]_\wr&\sN\otimes_{\sD_{X,\log}}\sD'_{X,\log,\leq k}\ar[d]\ar[r]&\gr_k(\sN\otimes_{\sD_{X,\log}}\sD'_{X,\log})\ar[d]\ar[r]&0\\
0\ar[r]&\sN x_1^{-k+1}\ar[r]&\sN x_1^{-k}\ar[r]&\sN x_1^{-k}/\sN x_1^{-k+1}\ar[r]&0
}
\]
and we aim at proving that the right vertical arrow is bijective. For that purpose, we consider the commutative diagram
\begin{equation}\label{eq:tensorloc}
\begin{array}{c}
\let\labelstyle\textstyle
\xymatrix{
\sN/\sN x_1\ar@{=}[d]\ar[r]^-{\partial_{x_1}^k}&\gr_k(\sN\otimes_{\sD_{X,\log}}\sD'_{X,\log})\ar[d]\\
\sN/\sN x_1\ar[r]^-{\partial_{x_1}^k}&\sN x_1^{-k}/\sN x_1^{-k+1}.
}
\end{array}
\end{equation}
By definition, the upper horizontal morphism is onto. We claim that it is also injective. We will check that the lower horizontal morphism is an isomorphism. This will imply the desired injectivity, hence the bijectivity of the upper horizontal arrow, together with that of the right vertical one.

Composing the lower horizontal morphism with right multiplication by $x_1^k$:
\[
\sN x_1^{-k}/\sN x_1^{-k+1}\To{x_1^k}\sN/\sN x_1,
\]
which is bijective, we find the morphism
\[
\partial_{x_1}^kx_1^k:\sN/\sN x_1\to\sN/\sN x_1,
\]
that we can write as $\prod_{j=1}^k(\Euler_1+j)$. Property \ref{prop:generalV}\eqref{prop:generalV2} for $I=\{1\}$ gives its bijectivity, as wanted.
\end{proof}

\begin{lemma}\label{lem:inductive}
Let $\sN$ be a $\sD_{X,\log}$-module satisfying \ref{prop:generalV}\eqref{prop:generalV1} and \eqref{prop:generalV2}. Then so does $\sN'=\sN(*D_1)$ as a $\sD'_{X,\log}$-module.
\end{lemma}

\begin{proof}
Let $I$ be a subset of $\{2,\dots,\ell\}$ and let $I^{\prime\rc}$ be its complement in $\{2,\dots,n\}$, while $I^\rc=\{1\}\cup I^{\prime\rc}$ is its complement in $\{1,\dots,n\}$. We wish to prove that the localization morphism
\[
i_I^*(\sN(*D_1))\to i_I^*(\sN(*D_1))(*D_I^{(I^{\prime\rc})})
\]
is injective. Since $x_1:\sN(*D_1)\to\sN(*D_1)$ is bijective, so is $x_1:i_I^*(\sN(*D_1))\to i_I^*(\sN(*D_1))$, and thus $i_I^*(\sN(*D_1))=(i_I^*\sN)(*D_1)$. Then, since the localization morphism $i_I^*\sN\to (i_I^*\sN)(*D_I^{(I^{\prime\rc})})$ is injective by \ref{prop:generalV}\eqref{prop:generalV1} for $\sN$, it remains so after applying the functor $(*D_1)$, which gives the desired injectivity, hence \ref{prop:generalV}\eqref{prop:generalV1} for $\sN'=\sN(*D_1)$. Similarly, \ref{prop:generalV}\eqref{prop:generalV2} is obtained by applying $(*D_1)$, since $\Euler_j$ commutes with $x_1$ for $j\neq1$.\end{proof}

\begin{proof}[End of the proof of Proposition \ref{prop:generalV}]
We argue by induction on $\#D$. We write
\[
\sN\otimes_{\sD_{X,\log}}\sD_X=\sN\otimes_{\sD_{X,\log}}\sD'_{X,\log}\otimes_{\sD'_{X,\log}}\sD_X=\sN'\otimes_{\sD'_{X,\log}}\sD_X.
\]
Due to Lemma \ref{lem:inductive} and the induction hypothesis, we can apply Proposition \ref{prop:generalV} to $\sN'$ with respect to the divisor $D'$, and we find
\[
\sN\otimes_{\sD_{X,\log}}\sD_X=\sN'(*D')=\sN(*D).\qedhere
\]
\end{proof}

\end{document}